\documentclass[11pt]{article}

\usepackage{amsmath, amsthm, amssymb} 
\usepackage{multirow}
\usepackage{psfrag, verbatim}
\usepackage {mathrsfs, graphicx, newcent}

\newcommand{\be}{\begin{equation}}
\newcommand{\ee}{\end{equation}}
\newcommand{\bea}{\begin{eqnarray}}
\newcommand{\eea}{\end{eqnarray}}
\newcommand{\beas}{\begin{eqnarray*}}
\newcommand{\eeas}{\end{eqnarray*}}

\newcommand{\bbE}{\mathbb E}
\newcommand{\bbF}{\mathbb F}

\newcommand{\bbH}{\mathbb H}

\newcommand{\bbM}{\mathbb M}
\newcommand{\bbN}{\mathbb N}

\newcommand{\bbP}{\mathbb P}

\newcommand{\bbR}{\mathbb R}
\newcommand{\bbS}{\mathbb S}

\newcommand{\scB}{\mathcal B}
\newcommand{\scC}{\mathcal C}

\newcommand{\scE}{\mathcal E}
\newcommand{\scF}{\mathcal F}

\newcommand{\scL}{\mathcal L}
\newcommand{\scM}{\mathcal M}

\newcommand{\scP}{\mathcal P}

\newcommand{\scW}{\mathcal W}

\newcommand{\ang}[1]{\ensuremath{ \left \langle #1 \right \rangle }}
\newcommand{\norm}[1]{\ensuremath{\left\| #1 \right\|}}
\newcommand{\abs}[1]{\ensuremath{\left| #1 \right|}}

\DeclareMathOperator*{\esssup}{ess\,sup}

\newcommand{\crl}[1]{\ensuremath{ \left\{ #1 \right\} }}
\newcommand{\edg}[1]{\ensuremath{ \left[ #1 \right] }}
\newcommand{\brak}[1]{\ensuremath{\left( #1 \right)}}

\newtheorem{theorem}{Theorem}[section]
\newtheorem{definition}[theorem]{Definition}
\newtheorem{proposition}[theorem]{Proposition}
\newtheorem{corollary}[theorem]{Corollary}
\newtheorem{lemma}[theorem]{Lemma}
\newtheorem{remark}[theorem]{Remark}
\newtheorem{example}[theorem]{Example}
\newtheorem{examples}[theorem]{Examples}
\newtheorem{foo}[theorem]{Remarks}
\newenvironment{Example}{\begin{example}\rm}{\end{example}}

\newenvironment{Remark}{\begin{remark}\rm}{\end{remark}}

\title{BSE'S, BSDE'S AND FIXED POINT PROBLEMS\footnote{We thank Ramon van Handel,
Ying Hu, Peter Imkeller, Shige Peng, and Frederi Viens for fruitful discussions and helpful comments.}}
\author{
Patrick Cheridito\\
ETH Zurich\\ 8092 Zurich,
Switzerland\\
\and
Kihun Nam\\
Rutgers University\\
Piscataway, NJ 08854, USA
}
\date{August 2016}

\begin{document}
\maketitle
\begin{abstract}
In this paper, we introduce a class of backward stochastic equations (BSEs) that extend 
classical BSDEs and include many interesting examples of generalized BSDEs as well as
semimartingale backward equations. We show that a BSE can be translated 
into a fixed point problem in a space of random vectors. This makes it possible to employ
general fixed point arguments to establish the existence of a solution. For instance, Banach's contraction 
mapping theorem can be used to derive general existence and uniqueness 
results for equations with Lipschitz coefficients, whereas Schauder-type fixed 
point arguments can be applied to non-Lipschitz equations.
The approach works equally well for multidimensional as for one-dimensional 
equations and leads to results in several interesting cases such as equations with 
path-dependent coefficients, anticipating equations, McKean--Vlasov type equations 
and equations with coefficients of superlinear growth.\\[2mm]
\end{abstract}

\setcounter{equation}{0}
\section{Introduction}\label{intro}
\label{sec:intro}

In this paper we study backward stochastic equations (BSEs) of the form 
\be \label{bseintro}
Y_t + F_t(Y,M) + M_t = \xi + F_T(Y,M) + M_T.
\ee
For a given maturity $T \in \mathbb{R}_+$, a filtered probability space $(\Omega, {\cal F}, ({\cal F}_t)_{0 \le t \le T}, \bbP)$,
a generator $F$ and a terminal condition $\xi \in L^p({\cal F}_T)^d$, a solution to \eqref{bseintro}
consists of a $d$-dimensional adapted process $Y$ together with a $d$-dimensional
martingale $M$ such that equation \eqref{bseintro} holds for all $t \in [0,T]$. 
If $F(Y,M)$ is a finite variation process, \eqref{bseintro} is a semimartingale 
backward equation, which as a special case, contains the semimartingale 
Bellman equation introduced by Chitashvili (1983); see also Mania and Tevzadze
(2003) and the references therein. In the case where $F$ is of the form
$F_t(Y,M) = \int_0^t f(s,Y,M) ds$, BSE \eqref{bseintro} becomes a
generalized backward 
stochastic differential equation (BSDE),
\be \label{genBSDE}
Y_t = \xi + \int_t^T f(s,Y,M) ds + M_T - M_t,
\ee
in the spirit of Liang et al. (2011). If in addition, the probability space 
carries an $n$-dimensional Brownian motion $W$ and a 
Poisson random measure $N$ on $[0,T] \times (\mathbb{R}^m \setminus \crl{0})$
such that every square-integrable martingale $M$ has a unique representation of the form 
$$
M_t= \int_0^t Z^M_s dW_s + \int_0^t \int_{\mathbb{R}^m \setminus \crl{0}} U^M_s(x) \tilde{N}(ds,dx) + K^M_t
$$
for the compensated Poisson random measure $\tilde{N}$, suitable integrands $Z^M$ and $U^M$,
and a square-integrable martingale $K^M$ strongly orthogonal to $W$ and $\tilde{N}$, 
one can write equations of the form
\be \label{bsdeWN}
Y_t = \xi + \int_t^T f(s,Y,Z^M,U^M) ds + M_T - M_t.
\ee
This generalizes the jump-diffusion extension of Tang and Li (1994) of the classical BSDEs introduced by Pardoux and Peng (1990) in three directions. 
First, in Tang and Li (1994) the filtration is generated by the Brownian motion and 
the Poisson random measure, whereas here it is general; secondly, at any given time,
the driver $f$ in \eqref{bsdeWN} can depend on the whole paths of the processes 
$Y$, $Z^M$, $U^M$ and not only on their current values; and finally, $f$
can be a function of $Y$, $Z^M$, $U^M$ viewed as random elements instead of just
their realizations $Y(\omega)$, $Z^M(\omega)$ and $U^M(\omega)$.
As special cases, \eqref{bsdeWN} contains BSDEs with drivers that depend on the
past or future of $Y$, $Z^M$ and $U^M$, such as the time-delayed BSDEs of 
Delong and Imkeller (2010a, 2010b) or the anticipating BSDEs of Peng and Yang (2009).
It also includes mean-field BSDEs as in Buckdahn et al. (2009), or more generally, 
McKean--Vlasov type BSDEs with coefficients depending on the distributions 
of $Y$, $Z^M$ and $U^M$.

Our approach to proving that a BSE has a solution is to translate it into a fixed 
point problem for a mapping $G : L^p({\cal F}_T)^d \to L^p({\cal
  F}_T)^d$. This makes it possible to apply general fixed point results. For instance,
Banach's contraction mapping theorem can be used to derive general 
existence and uniqueness results for equations with Lipschitz coefficients.
In the non-Lipschitz case one can employ Schauder type fixed point arguments.
This yields results for equations with coefficients of 
superlinear growth, but it requires compactness assumptions. By reducing a 
BSE to a fixed point problem in $L^p({\cal F}_T)^d$, one eliminates the time-dimension. 
But one still has to find compact subsets of $L^p({\cal F}_T)^d$. We do that 
by making use of Sobolev spaces corresponding to infinite-dimensional Gaussian measures.

Our method works equally well for multidimensional as for one-dimensional equations,
and in addition to general results for BSEs, it also yields interesting findings for BSDEs.
For instance, in Section \ref{sec:contr}, we obtain existence and uniqueness results for 
BSDEs with functional drivers depending on the whole processes $Y$ and $M$. 
In general, such results require Lipschitz continuity with a small enough Lipschitz 
constant or, alternatively, a sufficiently short maturity. But in several interesting 
special cases, it is possible to derive the existence of a unique solution for arbitrary 
Lipschitz constant and maturity. 
In Section \ref{sec:compact}, we use compactness and a theorem by Krasnoselskii (1964),
which combines the fixed point results of Banach and Schauder, to derive existence results 
for multidimensional BSDEs with functional drivers of superlinear growth. For instance, 
Corollary \ref{corcor:Lip} establishes the existence of solutions to BSDEs with 
general path-dependent drivers and Corollary \ref{cor:f1f2}  the existence of a solution 
to a multidimensional mean-field BSDE with driver of quadratic growth. The latter 
complements results by e.g., Tevzadze (2008) and Cheridito and Nam (2015) on 
multidimensional quadratic BSDEs, which are known to not always have 
solutions (see e.g., Peng, 1999, or Frei and dos Reis, 2011).

The structure of the paper is as follows. In Section \ref{sec:fp}, we formally introduce BSEs and 
relate them to fixed point problems in $L^p({\cal F}_T)^d$. 
In Section \ref{sec:contr}, we derive existence and uniqueness results for various BSEs
and BSDEs with general functional Lipschitz coefficients from Banach's contraction mapping 
theorem. In Section \ref{sec:compact}, we provide existence results for different non-Lipschitz 
equations using compactness and Krasnoselskii's fixed point theorem. 

\setcounter{equation}{0}
\section{BSEs and fixed points in $L^p$}
\label{sec:fp}

In this section, we introduce BSEs and show how they can be translated into fixed point
problems in $L^p$-spaces. We fix a finite time horizon $T \in \mathbb{R}_+$ 
and let $(\Omega, {\cal F},\bbF, \bbP)$ be a filtered probability space with a filtration 
$\bbF:=({\cal F}_t)_{t \in [0,T]}$ satisfying the usual conditions. Then all martingales admit a RCLL modification
(i.e., right-continuous with left limits).
By $|.|$ we denote the Euclidean norm on $\mathbb{R}^d$, and for a $d$-dimensional random vector $X$, we define 
$$
\norm{X}_p := (\bbE |X|^p)^{1/p} \mbox{ if } p < \infty \quad \mbox{and} \quad
\norm{X}_{\infty} := \esssup_{\omega\in\Omega}|X|.
$$
For $p \in (1,\infty]$, we set:

\begin{itemize}
\item $L^p({\cal F}_t)^d$: all $d$-dimensional $\scF_t$-measurable random vectors $X$ satisfying
$\norm{X}_p < \infty$

\item $\bbE_t X:=\bbE[X|\scF_t]$

\item 
$\mathbb{S}^p$: all $\mathbb{R}^d$-valued RCLL adapted processes
$(Y_t)_{0 \le t \le T}$ satisfying
$
\|Y\|_{\mathbb{S}^p}: = \norm{\sup_{0 \leq t \le T}|Y_{t}|}_p < \infty
$

\item $\bbS^p_0:$ all $Y \in \bbS^p$ with $Y_0 = 0$

\item 
$\bbM^p_0$: all martingales in $\bbS_0^p$.
\end{itemize}
A BSE is specified by a generator $F : \bbS^p \times \bbM^p_0 \to \bbS^p_0$
and a terminal condition $\xi \in L^p({\cal F}_T)^d$.

\begin{definition}
A solution to the BSE
\be \label{bse}
Y_t + F_t(Y,M) + M_t = \xi + F_T(Y,M) + M_T
\ee
consists of a pair $(Y,M) \in \bbS^p \times \bbM^p_0$ such that 
\eqref{bse} holds for all $t \in [0,T]$.
\end{definition}

\begin{definition}
We say $F$ satisfies condition {\rm (S)} if for all $y \in L^p({\cal F}_0)^d$ 
and $M \in \bbM^p_0$, the equation
\be \label{sde}
Y_t = y - F_t(Y,M) - M_t 
\ee
has a unique solution $Y \in \bbS^p$. 
\end{definition}

For a given $V \in L^p({\cal F}_T)^d$, one obtains from Jensen's inequality that 
$y^V := \bbE_0 V$ belongs to $L^p({\cal F}_0)^d$ and from Doob's $L^p$-maximal inequality that
$M^V_t := \bbE_0 V - \bbE_t V$ is in $\bbM^p_0$. If $F$ satisfies (S), we denote by $Y^V$ the solution 
of the equation $Y_t = y^V - F_t(Y,M^V) - M^V_t$. 

A BSE depends on the generator $F$ and terminal condition $\xi$.
Provided that $F$ satisfies condition (S), then the pair $(F, \xi)$ also defines a map
$$
G : L^p({\cal F}_T)^d \to L^p({\cal F}_T)^d \quad \mbox{through} \quad V \mapsto \xi + F_T(Y^V,M^V).
$$
To relate solutions of the BSE \eqref{bse} to fixed points of $G$, we define the two mappings
$$
\pi : \bbS^p \times \bbM^p_0 \to L^p({\cal F}_T)^d \quad \mbox{and} \quad
\phi : L^p({\cal F}_T)^d \to \bbS^p \times \bbM^p_0$$ by
$$
\pi(Y,M) := Y_0 - M_T \quad \mbox{and} \quad \phi(V) := (Y^V, M^V).
$$

\begin{theorem} \label{thm:fix} 
Assume $F$ satisfies {\rm (S)}. Then the following hold:
\begin{itemize}
\item[{\rm a)}]
$V = (\pi \circ \phi)(V)$ for all $V \in L^p({\cal F}_T)^d$. In particular, $\phi$ is injective.

\item[{\rm b)}]
If $V \in L^p({\cal F}_T)^d$ is a fixed point of $G$, then
$\phi(V)$ is a solution of the BSE \eqref{bse}.

\item[{\rm c)}]
If $(Y,M) \in \bbS^p \times \bbM^p_0$ solves 
the BSE \eqref{bse}, then $\pi(Y,M)$ is a fixed point of $G$ and \linebreak $(Y,M) = (\phi \circ \pi)(Y,M)$.

\item[{\rm d)}] 
$V$ is a unique fixed point of $G$ in $L^p({\cal F}_T)^d$ if and only if
$\phi(V)$ is a unique solution of the BSE \eqref{bse} in $\bbS^p \times \bbM^p_0$.
\end{itemize}
\end{theorem}

\begin{proof} a) is straight-forward to check.

b) If $V \in L^p({\cal F}_T)^d$ is a fixed point of $G$, then 
\be \label{beq}
y^V - M^V_T = (\pi \circ \phi)(V) = V =  G(V) = \xi + F_T(Y^V,M^V).
\ee
Since $Y^V$ satisfies $Y^V_t = y^V - F_t(Y^V,M^V) - M^V_t$ for all $t$,
\eqref{beq} is equivalent to 
$$
Y_t^V + F_t(Y^V,M^V) + M^V_t = \xi + F_T (Y^V,M^V) + M^V_T \quad \mbox{for all } t,
$$
which shows that $\phi(V) = (Y^V, M^V)$ solves the BSE \eqref{bse}.

c) Let $(Y,M) \in \bbS^p \times \bbM^p_0$ be a solution of the 
BSE \eqref{bse}. Set $V:=\pi(Y,M)=Y_0-M_T$. Then, $y^V=Y_0$ and $M^V_t=M_t$.
In particular,
\[
Y_t = Y_0  - F_t(Y,M) - M_t=y^V-F_t(Y,M^V)-M^V_t
\] for all $t$. It follows that $(Y,M) = (Y^V,M^V) = \phi(V) =(\phi \circ \pi)(Y,M)$ and 
$$
y^V=Y^V_0 = \xi + F_T (Y^V,M^V) + M^V_T=G(V)+M^V_T.
$$
Since $y^V-M_T^V=V$, this shows that $V = G(V)$.

d) follows from a)--c).
\end{proof}

In the special case, where $F$ does not depend on $Y$, condition (S) 
holds trivially, and it is enough to find a fixed point of the 
mapping $G_0(V) := G(V) - \mathbb{E}_0 G(V)$ in the
subspace $$L^p_0({\cal F}_T)^d := \crl{V \in L^p({\cal F}_T)^d : \mathbb{E}_0 V = 0}.$$

\begin{corollary} \label{cor:fix}
If $F$ does not depend on $Y$, the following hold:
\begin{itemize}
\item[{\rm a)}]
If $V \in L^p_0({\cal F}_T)^d$ is a fixed point of $G_0$, then the processes 
$Y_t := \bbE_0 \xi + \bbE_0 F_T(M) -F_t(M) - M_t$ and 
$M_t := - \bbE_t V$ form a solution of the BSE
\eqref{bse} in $\bbS^p \times \bbM^p_0$.

\item[{\rm b)}]
If $(Y,M) \in \bbS^p \times \bbM^p_0$ solves the BSE \eqref{bse}, then 
$- M_T$ is a fixed point of $G_0$.

\item[{\rm c)}] 
$V$ is a unique fixed point of $G_0$ in $L^p_0({\cal F}_T)^d$ if and only if the pair 
$(Y,M)$ given by $Y_t := \bbE_0 \xi + \bbE_0 F_T(M) -F_t(M) - M_t$ and 
$M_t := - \bbE_t V$ is a unique solution of the BSE
\eqref{bse} in $\bbS^p \times \bbM^p_0$.
\end{itemize}
\end{corollary}

\begin{proof}
a) If $V = G_0(V)$, then for $\tilde{V} = V + \bbE_0 G(V)$, 
one has $M^{\tilde{V}} = M^V$, and therefore,
$$
\tilde{V} = V + \bbE_0 G(V)= G(V) =\xi+F_T(M^{V}) =\xi+F_T(M^{\tilde{V}})= G(\tilde{V}).
$$
So it follows from Theorem \ref{thm:fix} that the pair $(Y,M)$ given by 
$Y_t := \bbE_0 \xi + \bbE_0 F_T(M) - F_t(M) - M_t$ and 
$M_t := - \bbE_t V$ solves the BSE \eqref{bse}.

b) If $(Y,M) \in \bbS^p \times \bbM^p_0$ solves the BSE \eqref{bse}, it follows from 
Theorem \ref{thm:fix} that $V:=Y_0 - M_T$ is a fixed point of $G$. So
\[
G_0(-M_T)=G_0(Y_0-M_T)=G(V)-\bbE_0G(V)=V-\bbE_0V=-M^V_T=-M_T.
\]

c) $V$ is a fixed point of $G_0$ if and only if $V + \bbE_0 G(V)$ is a fixed point 
of $G$. Therefore, the result follows from part d) of Theorem \ref{thm:fix}.
\end{proof}

The following lemma provides a sufficient condition for $F$ to satisfy condition (S).
For $(Y,M) \in \bbS^p \times \bbM^p_0$ and $k \in \bbN$, define
$$
F^{(k)}_t(Y,M) := F_t(Y^{(k,M)},M),
$$
where $Y^{(k,M)}$ is recursively given by
$$
Y^{(1,M)} := Y \quad \mbox{and} \quad Y^{(k,M)}_t := Y_0 - F_t(Y^{(k-1,M)},M) - M_t, \quad 
k \ge 2.
$$

\begin{lemma} \label{lemma:LipY}
If for given $y \in L^p({\cal F}_0)^d$ and $M \in \bbM^p_0$, there 
exist a number $k \in \mathbb{N}$ and a constant $C <1$ such that
\be \label{FkLip}
\norm{F^{(k)}(Y,M)-F^{(k)}(Y',M)}_{\bbS^p} \le C \norm{Y-Y'}_{\bbS^p} \quad \mbox{for all }
Y,Y' \in \bbS^p \mbox{ with } Y_0 = Y'_0 = y,
\ee
then the SDE \eqref{sde} has a unique solution $Y \in \bbS^p$.
\end{lemma}

\begin{proof}
The mapping $Y \mapsto y - F^{(k)}(Y,M) - M$ is a contraction on
$\crl{Y \in \mathbb{S}^p : Y_0 = y}$.
So it follows from Banach's contraction mapping theorem that there exists a unique 
$Y \in \bbS^p$ satisfying $Y = y - F^{(k)}(Y,M) - M =
Y^{(k+1,M)}$. This implies 
$$
Y^{(2,M)} = y-F_t(Y,M)-M_t=y-F_t(Y^{(k+1,M)},M)-M_t= Y^{(k+2,M)} = y - F^{(k)}(Y^{(2,M)},M) - M,
$$
from which one deduces $Y = Y^{(2,M)} = y - F(Y,M) - M$. This shows that $Y$ solves 
the SDE \eqref{sde}. If $Y' \in \bbS^p$ is another solution of 
\eqref{sde}, then $Y' = y - F^{(k)}(Y',M) - M$, and one obtains $Y' = Y$. 
\end{proof}

\setcounter{equation}{0}
\section{Existence and uniqueness of solutions under Lipschitz assumptions}
\label{sec:contr}

In this section we consider equations with Lipschitz coefficients and
use Banach's contraction mapping theorem to show that they 
have unique solutions.

\subsection{General existence and uniqueness results}

We start with a result for general Lipschitz BSEs. Let us denote
$$
c_2 = \frac{1}{5}, \quad c_{\infty} = \frac{1}{4} \quad \mbox{and} \quad c_p = \frac{p-1}{4p-1} \quad
\mbox{for } p \in (1,\infty) \setminus \crl{2}.
$$
Then the following holds:

\begin{theorem} \label{thm:contra}
Let $\xi\in L^p(\scF_T)^d$ for some $p \in (1, \infty]$. If there exist a number 
$k \in \mathbb{N}$ and a constant $C < c_p$ such that
\begin{align*}
\norm{F^{(k)}(Y,M)-F^{(k)}(Y',M')}_{\bbS^p}\leq C \brak{\norm{Y-Y'}_{\bbS^p}+\norm{M-M'}_{\bbS^p}}
\quad \mbox{for all} \quad Y,Y' \in \mathbb{S}^p \mbox{ and } M,M' \in \mathbb{M}^p_0,
\end{align*}
then the BSE \eqref{bse} has a unique solution $(Y,M)$ in $\bbS^p \times \bbM^p_0$.
\end{theorem}

\begin{proof}
Since $C < 1$, it follows from Lemma \ref{lemma:LipY} that $F$ satisfies (S). 
So by Theorem \ref{thm:fix}, it is enough to prove that $G$ has a unique fixed 
point in $L^p({\cal F}_T)^d$. This follows from Banach's contraction mapping theorem
if we can show that $G$ is a contraction on $L^p({\cal F}_T)^d$. Since for $V \in L^p(\scF_T)^d$, 
$Y^V$ is the unique fixed point of the mapping $Y \mapsto \bbE_0V-F(Y,M^V)-M^V$, 
it follows from the definition of $F^{(k)}$ that $F(Y^V,M^V)=F^{(k)}(Y^V,M^V)$.
Hence, one has for all $V,V' \in L^p({\cal F}_T)^d$, 
\beas
Y^V_t- Y^{V'}_t &=& y^V - y^{V'} - \crl{F^{(k)}_t(Y^V, M^V) - F^{(k)}_t(Y^{V'}, M^{V'})}
- (M^V_t - M^{V'}_t)\\
&=& \bbE_t(V-V') - \crl{F^{(k)}_t(Y^V, M^V) - F^{(k)}_t(Y^{V'}, M^{V'})}.
\eeas 
Therefore,
\begin{align*}
\sup_{0\leq t\leq T}|Y^V_t- Y^{V'}_t| & \leq \sup_{0 \leq t \leq T}|
\bbE_t(V-V')|+\sup_{0\leq t\leq T}|F^{(k)}_t(Y^V, M^V) - F^{(k)}_t(Y^{V'}, M^{V'})|,
\end{align*}
and it follows that
\begin{align*}
\norm{Y^V - Y^{V'}}_{\bbS^p}& \leq \norm{\sup_{0 \leq t \leq T}|\bbE_t(V-V')|}_p
+\norm{F^{(k)}(Y^V, M^V)-F^{(k)}(Y^{V'},M^{V'})}_{\bbS^p}\\
&\leq\norm{\sup_{0\leq t\leq T}|\bbE_t(V - V')|}_p 
+C \brak{\norm{Y^V - Y^{V'}}_{\bbS^p}+ \norm{M^V - M^{V'}}_{\bbS^p}}.
\end{align*}
In particular,
\begin{align*}
\norm{Y^V - Y^{V'}}_{\bbS^p}& \leq\frac{1}{1-C} \brak{\norm{\sup_{0\leq t\leq T}
|\bbE_t(V-V')|}_p+ C \norm{M^V - M^{V'}}_{\bbS^p}},
\end{align*}
and therefore,
\begin{align*}
& \norm{G(V)-G(V')}_p = \norm{F^{(k)}_T(Y^V,M^V)-F^{(k)}_T(Y^{V'}, M^{V'})}_p
\leq C \brak{\norm{Y^V-Y^{V'}}_{\bbS^p}+\norm{M^V-M^{V'}}_{\bbS^p}}\\
& \le \frac{C}{1-C}\brak{\norm{\sup_{0\leq t\leq T}|\bbE_t(V-V')|}_p
+C \norm{M^V - M^{V'}}_{\bbS^p}} + C \norm{M^V - M^{V'}}_{\bbS^p}\\
&= \frac{C}{1-C}\brak{\norm{\sup_{0\leq t\leq T}|\bbE_t(V-V')|}_p+\norm{M^V - M^{V'}}_{\bbS^p}}.
\end{align*}
By Doob's $L^p$-maximal inequality, if we let $C_{p}=p/(p-1)$ for
$p\in(1,\infty)$ and $C_{\infty}=1$,
$$
\norm{\sup_{0\leq t\leq T}|\bbE_t(V-V') - \bbE_0(V-V')|}_p 
\leq C_{p} \norm{V-V' - \bbE_0(V-V')}_p,
$$
and
$$
\norm{\sup_{0\leq t\leq T}|\bbE_t(V-V')|}_p \leq C_{p}\norm{V-V'}_p.
$$
Hence,
$$
\norm{M^V - M^{V'}}_{\bbS^p} \le
\left\{
\begin{array}{cc}
2 \norm{V-V' - \bbE_0(V-V')}_2
\le 2 \norm{V-V'}_2 & \mbox{ for } p =2\\
C_{p} \norm{V-V' - \bbE_0(V-V')}_p
\leq 2C_{p} \norm{V-V'}_p & \mbox{ for } p \neq 2
\end{array}
\right.,
$$
and
$$
\norm{G(V)-G(V')}_p \le \left\{
\begin{array}{cc}
\frac{4C}{1-C}  \norm{V-V'}_2 & \mbox{ for } p =2\\
3C_p \frac{C}{1-C} \norm{V-V'}_p & \mbox{ for } p \neq 2
\end{array} \right..
$$
This shows that $G$ is a contraction.
\end{proof}

\begin{Remark} \label{rem:counter}
One cannot hope to obtain a general existence and uniqueness result like Theorem \ref{thm:contra} for equations with 
path-dependent coefficients without the assumption that the Lipschitz constant $C$ is sufficiently small. 
For instance, if the generator is given by $F_t(Y,M) = atY_0 $ for a constant $a$, the BSE \eqref{bse} 
takes the form
\be \label{counterex}
Y_t - a(T-t)Y_0 = \xi + M_T - M_t.
\ee
This is a variant of the equation studied in Example 3.1 of Delong and Imkeller (2010a),
who noticed that time-delayed BSDEs with Lipschtitz coefficients are not always well-posed. 
Obviously, $F(Y,M)$ is Lipschitz in $(Y,M)$. But if 
one sets $t=0$ and takes expectation on both sides of \eqref{counterex}, one obtains
$(1- aT)Y_0 = \mathbb{E}_0 \xi$. This shows that for
$aT = 1$ and $\mathbb{E}_0 \xi \neq 0$, \eqref{counterex} cannot have a solution. On the other hand, 
if $aT = 1$ and $\mathbb{E}_0 \xi = 0$ then $Y_t = (1- t/T)Y_0 + \mathbb{E}_t \xi$ and
$M_t = - \mathbb{E}_t \xi$ defines a solution for any initial value $Y_0 \in L^p({\cal F}_0)^d$.
So in this case, \eqref{counterex} has infinitely many solutions in $\bbS^p \times \bbM^p_0$.
\end{Remark}
	
If the generator is of integral form $F_t(Y,M) = \int_0^t f(s,Y,M) ds$ for a driver
\be \label{gendriver}
f :[0,T] \times \Omega \times \bbS^p \times \bbM^p_0 \to \bbR^d,
\ee
the BSE \eqref{bse} becomes a BSDE of the general form
\be \label{BSDEYM}
Y_t = \xi + \int_t^T f(s,Y, M)ds + M_T-M_t.
\ee
If for a RCLL measurable processe $X$, one denotes
$\norm{X}_{\bbS^p_{[0,t]}} := \norm{\sup_{0 \le s \le t} |X_t|}_p$, the following holds:

\begin{proposition} \label{prop:genBSDE}
Let $\xi \in L^p({\cal F}_T)^d$ for some $p \in (1, \infty]$. Then 
the BSDE \eqref{BSDEYM} has a unique solution 
$(Y,M) \in \bbS^p \times \bbM^p_0$ for every driver of the form \eqref{gendriver} 
satisfying the following conditions:
\begin{itemize}
\item[{\rm (i)}] For all $(Y,M) \in \bbS^p \times \bbM^p_0$, $f(\cdot,Y,M)$ is progressively measurable
with $\int_0^T \norm{f(t,0,0)}_p dt < \infty$.
\item[{\rm (ii)}] There exist nonnegative constants
$$
C_1 > 0 \quad \mbox{and} \quad C_2 < \frac{c_pC_1}{e^{C_1T}-1}
$$
such that
\beas
&& \norm{f(t,Y,M)-f(t,Y',M')}_p\\
&\le& C_1 \norm{Y - Y_0 + M - (Y' - Y'_0 + M')}_{\bbS^p_{[0,t]}}
 + C_2 \brak{\norm{Y_0 - Y'_0}_p + \norm{M-M'}_{\bbS^p}}
\eeas
for all $(Y,M),(Y',M') \in \bbS^p \times \bbM^p_0$.
\end{itemize}
\end{proposition}

\begin{proof}
Let $q= p/(p-1) \in [1,\infty)$. It follows from the assumptions that for all $(Y,M) \in \bbS^p \times \bbM^p_0$,
\beas
&& \norm{\int_0^T |f(t,Y, M)| dt}_p = \sup_{\norm{X}_q \leq 1}\int_0^T \bbE\edg{|f(t,Y,M)| |X|}dt\\
&\le& \sup_{\norm{X}_q\leq 1}\int_0^T\norm{f(t,Y,M)}_p \norm{X}_q dt = \int_0^T \norm{f(t,Y, M)}_p dt\\
&\le& \int_0^T \norm{f(t,0,0)}_p dt + T C_1
 \norm{Y - Y_0 + M}_{\bbS^p} + T C_2 \brak{\norm{Y_0}_p + \norm{M}_{\bbS^p}} < \infty.
\eeas
So $F_t(Y,M) := \int_0^t f(s,Y,M) ds$ is a well-defined mapping from 
$\bbS^p \times \bbM^p_0$ to $\bbS^p_0$ for all $p\in(1,\infty]$.

For given $Y,Y' \in \bbS^p$ and $M,M' \in \bbM^p_0$, set
\begin{align*}
\delta&:=\frac{C_2}{C_1} \brak{\norm{Y_0-Y'_0}_p + \norm{M-M'}_{\bbS^p}}\\
H^0_t &:= H^0 := 2 \brak{\norm{Y-Y'}_{\bbS^p} + \norm{M-M'}_{\bbS^p}}\\
H^k_t &:=\norm{F^{(k)}(Y,M)-F^{(k)}(Y',M')}_{\bbS^p_{[0,t]}}.
\end{align*}
Then 
\beas
H^k_t &\le& \int_0^t \norm{f(s,Y^{(k,M)}, M)-f(s,(Y')^{(k,M')}, M')}_pds\\
&\le& \int_0^t \brak{C_1 H^{k-1}_s
+C_2 \brak{\norm{Y_0-Y'_0}_p + \norm{M-M'}_{\bbS^p}}} ds\\
&\le& C_1 \int_0^t (H^{k-1}_s + \delta) ds,
\eeas
and by iteration,
\begin{align*}
H^k_t & \leq \frac{(C_1t)^k}{k!} H^0 + \brak{C_1 t + \dots + \frac{(C_1 t)^k}{k!}} \delta.
\end{align*}
In particular,
\beas
&& \norm{F^{(k)}(Y,M)-F^{(k)}(Y',M)}_{\bbS^p}\\
&\le& 2 \frac{(C_1T)^k}{k!} \brak{\norm{Y-Y'}_{\bbS^p} + \norm{M-M'}_{\bbS^p}}
+ \brak{e^{C_1 T} - 1} \frac{C_2}{C_1} \brak{\norm{Y_0-Y'_0}_p + \norm{M-M'}_{\bbS^p}}.
\eeas
So for $k$ large enough, there exists a constant $C < c_p$ such that
$$
\norm{F^{(k)}(Y,M) - F^{(k)}(Y',M')}_{\bbS^p} \le C \brak{\norm{Y-Y'}_{\bbS^p}+ \norm{M-M'}_{\bbS^p}},
$$
and the proposition follows from Theorem \ref{thm:contra}.
\end{proof}

\begin{Remark}
The backward stochastic dynamics
$$
Y_t = \int_t^T f_0(s,Y_s,L(M)_s)ds + \int_t^T f(s,Y_s) dB_s - (M_T-M_t)
$$
studied by Liang et al. (2011) can be viewed as a BSE with generator
$$
F_t(Y,M) = \int_0^t f_0(s,Y_s,L(M)_s)ds + \int_0^t f(s,Y_s) dB_s.
$$
But it also fits into the framework \eqref{BSDEYM} if the transformation
$$
\tilde{M}_t = \int_0^t f(s,Y_s) dB_s - M_t 
\quad \mbox{and} \quad
\tilde{f}(t,Y,\tilde{M}) = f_0\brak{t,Y_t, L\brak{\int f(s,Y_s) dB_s - \tilde{M}}_t}
$$
is applied. In addition, \eqref{BSDEYM} includes BSDEs with drivers depending on the past or future
of the processes $Y$ and $M$, such as the time-delayed BSDEs of 
Delong and Imkeller (2010a, 2010b) or the anticipating BSDEs of Peng and Yang (2009).
Previous existence and uniqueness results like Theorem 3.3 of Liang et al. (2011), 
Theorem 2.1 of Delong and Imkeller (2010a) or Theorem 2.1 of Delong and 
Imkeller (2010b), can all be recovered as special cases of Proposition \ref{prop:genBSDE}.
\end{Remark}

\begin{Remark} \label{rmkpath}
Let $f : [0,T] \times \Omega \times \bbS^p \times \bbM^p_0 \to\bbR^d$ be a driver
satisfying condition (i) of Proposition \ref{prop:genBSDE} for some $p \in (1,\infty]$.
If there exist nonnegative constants $D_1, D_2$ such that
$$
\norm{f(t,Y,M) - f(t,Y',M')}_p \le D_1 \norm{Y-Y'}_{\bbS^p_{[0,t]}} + D_2 \norm{M-M'}_{\bbS^p}
$$
for all $Y,Y' \in \bbS^p$ and $M,M' \in \bbM^p_0$, then  
\beas
&&
\norm{f(t,Y,M) - f(t,Y',M')}_p\\
&\le& D_1 \norm{Y - Y_0 + M - (Y'-Y'_0+M')}_{\bbS^p_{[0,t]}} + D_1 \norm{Y_0-Y'_0}_p
+ (D_1 + D_2) \norm{M-M'}_{\bbS^p}.
\eeas
So the assumptions of Proposition \ref{prop:genBSDE} only hold if 
the constants $D_1$ and $D_2$ are small enough, or alternatively, the maturity $T$ 
is sufficiently short. This is in line with Remark \ref{rem:counter} above (note that 
\eqref{counterex} is a path-dependent BSDE of the form \eqref{BSDEYM} with $f(t,Y,M) = a Y_0$).
\end{Remark}

The following corollary gives conditions under which it directly follows from Proposition 
\ref{prop:genBSDE} that the BSDE \eqref{BSDEYM} has a unique solution for 
arbitrary Lipschitz constant and maturity. More examples of \eqref{BSDEYM}  
admitting solutions under general Lipschitz assumptions are given in 
Section \ref{subsec:BSDEBrownPoiss} below.

\begin{corollary}\label{cor:pathdep}
Let $p \in (1, \infty]$ and consider a terminal condition
$\xi \in L^{p}(\scF_{T})^{d}$ together with a driver $f$ 
of the form \eqref{gendriver} fulfilling condition {\rm (i)} of Proposition \ref{prop:genBSDE} such that
$f(t,Y,M) = h(t, Y- Y_0 + M)$ for a mapping $h : [0,T] \times \Omega \times \bbS^p_0 \to\bbR^d$.
If
$$
\norm{h(t,X) - h(t,X')}_p \le C \norm{X-X'}_{\bbS^p_{[0,t]}}, \quad X,X' \in \bbS^p_0
$$
for a constant $C \ge 0$, then the BSDE \eqref{BSDEYM} has a unique solution $(Y,M)\in\bbS^{p}\times\bbM^{p}_{0}$.
\end{corollary}

\subsection{Generalized Lipschitz BSDEs based on a Brownian motion and a Poisson random measure}
\label{subsec:BSDEBrownPoiss}

Let $W$ be an $n$-dimensional Brownian motion and $N$ an independent Poisson random measure 
on $[0,T] \times E$ for $E = \bbR^m \setminus \crl{0}$ with an
intensity measure of the form $dt \mu(dx)$ for a measure $\mu$ over 
the Borel $\sigma$-algebra ${\cal B}(E)$ of $E$ satisfying
$$
\int_E (1 \wedge |x|^2) \mu(dx) < \infty.
$$
Denote by $\tilde{N}$ the compensated random measure $N(dt,dx) - dt \mu(dx)$, and assume
that, for $A \in {\cal B}(E)$ with $\mu(A) < \infty$, $\tilde{N}([0,t] \times A)$ and $W$ are martingales with respect to $\bbF$. We need the following spaces of integrands:

\begin{itemize}
\item
$\mathbb{H}^2$: all $\mathbb{R}^{d \times n}$-valued predictable processes $Z$ satisfying
$$
\norm{Z}_{\mathbb{H}^2} := \brak{\int_0^T \bbE |Z_t|^2 dt}^{1/2} < \infty.
$$ 
\item
$L^2(\tilde{N})$: all $\scP \otimes {\cal B}(E)$-measurable 
mappings $U: [0,T] \times \Omega \times E \to \mathbb{R}^d$ such that
$$
\norm{U}_{L^2(\tilde{N})}:=\brak{\int_0^T \int_E \bbE |U_t(x)|^2 \mu(dx)dt}^{1/2} < \infty,
$$
where $\scP$ is the $\sigma$-algebra of $\bbF$-predictable
subsets of $[0,T] \times \Omega$. 
\end{itemize}
Any square-integrable $\bbF$-martingale $M \in \bbM^2_0$ has a unique representation of the form
\be \label{intrep}
M_t= \int_0^t Z^M_s dW_s + \int_0^t \int_E U^M_s(x) \tilde{N}(ds,dx) + K^M_t
\ee
for a triple $(Z^M,U^M,K^M) \in\bbH^2 \times L^2(\tilde{N}) \times\bbM^2_0$
such that $K^M$ is strongly orthogonal to $W$ and $\tilde{N}$ (see e.g. Jacod, 1979). 
This makes it possible to consider BSDEs 
\be \label{bsdeYZU}
Y_t= \xi + \int_t^T f(s,Y,Z^M,U^M)ds + M_T-M_t
\ee
for terminal conditions $\xi \in L^2(\scF_T)^d$ and drivers
\be \label{driverYZU}
f:[0,T] \times \Omega \times \bbS^2 \times \bbH^2 \times L^2(\tilde{N})\to\bbR^d.
\ee
In the special case where the filtration $\bbF$ is generated by $W$ and $N$, 
the orthogonal part $K^M$ in the representation \eqref{intrep} vanishes
(see e.g. Ikeda and Watanabe, 1989), and as a result, \eqref{bsdeYZU} can be written as
\be \label{standbsde}
Y_t= \xi + \int_t^T f(s,Y,Z^M,U^M)ds +  
\int_t^T Z^M_s dW_s + \int_t^T \int_E U^M_s(x) \tilde{N}(ds,dx).
\ee
This generalizes the classical BSDEs of Pardoux and Peng (1990) and Tang and Li (1994),
which have drivers that at time $s$ only depend on the realizations 
$Y_s(\omega)$, $Z^M_s(\omega)$, $U^M_s(\omega)$, to equations with functional drivers 
that can depend on the full processes $Y$, $Z^M$ and $U^M$.

In the rest of this subsection, we consider different specifications of \eqref{bsdeYZU} with drivers
depending on the future, present or past of the processes $Y$, $Z^M$ and $U^M$. In all instances, we
are able to derive the existence of a unique solution for an arbitrary Lipschitz constant and maturity. 
In the following proposition, the driver can depend on the present and future of $Y$, $Z^M$ and $U^M$, 
but not on their past -- this is ruled out by condition (ii). For its proof, we need the isometry
\be \label{isometry}
\bbE |M_t|^2 = \int_0^t \bbE |Z^M_s|^2 ds+\int_0^t \int_E \bbE |U^M_s(x)|^2 \mu(dx)ds + \bbE |K^M_t|^2
\ee
(see e.g. Jacod, 1979). 

\begin{proposition} \label{prop:anti}
The BSDE \eqref{bsdeYZU} has a unique solution $(Y,M) \in \bbS^2 \times \bbM^2_0$ for every terminal condition 
$\xi \in L^2(\scF_T)^d$ and driver
$$
f: [0,T] \times \Omega \times \bbS^2 \times \bbH^2 \times L^2(\tilde{N}) \to \bbR^d
$$
satisfying the following two conditions:
\begin{itemize}
\item[{\rm (i)}]
For all $(Y,Z,U) \in \mathbb{S}^2 \times \mathbb{H}^2\times L^2(\tilde{N})$, 
$f(t,Y, Z,U)$ is progressively measurable with \newline
$\int_0^T \norm{f(t,0,0,0}_2 dt < \infty$.
\item[{\rm (ii)}] There exists a constant $C \ge 0$ such that
$$
\int_t^T \norm{f(s,Y,Z, U) - f(s,Y',Z',U')}_2 ds
\le C \int_t^T \norm{Y_s -Y'_s}_2+\norm{Z_s -Z'_s}_2 + \norm{U_s-U'_s}_{L^2(\bbP \otimes \mu)} ds
$$
for all $t \in[0,T]$ and $(Y,Z,U),(Y',Z',U')\in \bbS^2\times\bbH^2\times L^2(\tilde{N})$.
\end{itemize} 
\end{proposition}

\begin{proof}
Choose $\delta > 0$ so that
\[
C\sqrt{3\delta(\delta +1)} < \frac{1}{5} \quad \text{ and }
\quad k :=T/\delta \in \bbN.
\]
By \eqref{isometry}, one has for every $M \in \bbM^2_0$,
\beas
&& \brak{\int_0^t \norm{Z^M_s}_{2} + \norm{U^M_s}_{L^2(\bbP \otimes \mu)} ds}^2
\le t \int_0^t \brak{\norm{Z^M_s}_{2} + \norm{U^M_s}_{L^2(\bbP \otimes \mu)}}^2 ds\\
&& \le 2 t \int_0^t \norm{Z^M_s}^2_{2} + \norm{U^M_s}^2_{L^2(\bbP \otimes \mu)} ds
\le 2 t \norm{M_t}^2_2.
\eeas
Therefore, one obtains from the assumptions for all $(Y,M) \in \bbS^2 \times \bbM^2_0$,
\beas
&& \norm{\int_{T-\delta}^T |f(s,Y, Z^M,U^M)| ds}_2 \le 
\int_{T-\delta}^T \norm{f(s,Y, Z^M, U^M)}_2 ds\\
&\le& \int_{T-\delta}^T \norm{f(s,0,0,0)}_2 ds + C \int_{T-\delta}^T \brak{\norm{Y_s}_2 +
\norm{Z^M_s}_2 + \norm{U^M_s}_{{L^2(\bbP \otimes \mu)}}} ds < \infty,
\eeas
where the first inequality follows from the same argument as in the proof of Proposition
\ref{prop:genBSDE}. In particular, for every pair $(Y,M) \in \bbS^2 \times \bbM^2_0$,
$$
F_t(Y,M) :=\int_0^t f(s,Y,Z^M,U^M)1_{[T-\delta, T]}(s)ds
$$ 
defines a process in $\bbS^2_0$. Furthermore, one has 
\begin{align*}
& \norm{F(Y,M)-F(Y',M')}_{\bbS^2}\\
\le &  \norm{\int_{T-\delta}^T \abs{f(s,Y,Z^M,U^M)-f(s,Y',Z^{M'}, U^{M'})} ds}_2\\
\le &  \int_{T-\delta}^T \norm{f(s,Y,Z^M,U^M) -f(s,Y',Z^{M'}, U^{M'})}_2 ds\\
\le &  C \int_{T-\delta}^T \norm{Y_s-Y_s'}_2
+ \norm{Z^M_s-Z^{M'}_s}_2 + \norm{U^M_s - U^{M'}_s}_{L^2(\bbP \otimes \mu)} ds\\
\le & C \sqrt{\delta \int_{T-\delta}^T \brak{\norm{Y_s-Y_s'}_2
+ \norm{Z^M_s-Z^{M'}_s}_2 + \norm{U^M_s - U^{M'}_s}_{L^2(\bbP \otimes \mu)}}^2 ds}\\
\le & C \sqrt{3 \delta \int_{T-\delta}^T \norm{Y_s-Y_s'}^2_2
+ \norm{Z^M_s-Z^{M'}_s}^2_2 + \norm{U^M_s - U^{M'}_s}^2_{L^2(\bbP \otimes \mu)} ds}\\
\le & C \sqrt{3 \delta^2 \norm{Y-Y'}^2_{\bbS^2} + 3 \delta \norm{M - M'}^2_{\bbS^2}}\\
\le & C \sqrt{3 \delta(\delta +1)} (\norm{Y-Y'}_{\bbS^2} + \norm{M - M'}_{\bbS^2})
\end{align*}
for all $(Y,M), (Y',M') \in\bbS^2\times\bbM^2_0$. Since $C \sqrt{3 \delta(\delta + 1)} < 1/5$,
one obtains from Theorem \ref{thm:contra} that the BSDE
$$
Y_t = \xi + \int_t^T f(s,Y,Z^M,U^M)
1_{[T-\delta, T]}(s) ds + M_T-M_t
$$
has a unique solution $(Y^{(k)},M^{(k)})$ in $\bbS^2 \times\bbM^2_0$. Now, consider the BSDE
\be \label{BSDEdelta}
Y_t = Y^{(k)}_{T-\delta} + 
\int_t^{T-\delta} f^{(k-1)}(s,Y,Z^M,U^M) 1_{[T-2\delta, T-\delta]}(s) ds +M_{T-\delta}-M_t
\ee
on the time interval $[0,T-\delta]$, where $f^{(k-1)}$ is given by 
\begin{align*}
f^{(k-1)}(s,Y,Z,U):= f \bigl( s, (Y,Z,U) 1_{[0, T-\delta)} +
\bigl(Y^{(k)},Z^{M^{(k)}}, U^{M^{(k)}}\bigr)1_{[T-\delta,T]}
\bigr). 
\end{align*}
Then the conditions (i)--(ii) still hold. So \eqref{BSDEdelta} has a unique solution $(Y^{(k-1)},M^{(k-1)})$ in $\bbS^2 \times\bbM_0^2$ over the time interval $[0,T-\delta]$.
Repeating the same argument, one obtains solutions $(Y^{(j)},M^{(j)})$, $j=1, \dots, k$. If one sets
$Y_t := Y^{(1)}_t$, $M_t := M^{(1)}_t$ for $0 \le t \le \delta$ and 
$Y_t := Y^{(j)}_t$, $M_t - M_{(j-1) \delta} := M^{(j)}_t - M^{(j)}_{(j-1) \delta}$
for $(j-1) \delta < t \le j \delta$, $j =2, \dots, k$, then $(Z^M_t, U^M_t) = (Z^{M^{(j)}}_t, U^{M^{(j)}}_t)$
for $(j-1) \delta < t \le j \delta$. Since this construction is backwards in time and by condition (ii), 
$f(t,Y,Z^M,U^M)$ cannot depend on the past of the processes $Y,Z^M$ and $U^M$, the pair
$(Y,M)$ forms a unique solution of \eqref{bsdeYZU} in $\bbS^2 \times \bbM^2_0$.
\end{proof}

\begin{Remark}
The assumptions of Proposition \ref{prop:anti} allow for drivers $f$ such that $f(t,Y,Z,U)$ depends on the future of the 
processes $Y, Z, U$ in a general ${\cal F}_t$-measurable way. This covers BSDEs with anticipating drivers of the form
$$
\begin{aligned}
- dY_t &= f(t,Y_t, Z_t, \mathbb{E}_t Y_{t + \delta(t)}, \mathbb{E}_t Z_{t + \zeta(t)}) dt + Z_t dW_t, \quad && t \in [0,T]\\
(Y_t,Z_t) &= (\xi_t,\eta_t), && t \in [T,T+K]
\end{aligned}
$$
or more generally,
\be \label{anti}
\begin{aligned}
- dY_t &= f(t,Y_t,Z_t, Y_{t + \delta(t)}, Z_{t + \zeta(t)}) dt + Z_t dW_t, \quad && t \in [0,T]\\
(Y_t,Z_t) &= (\xi_t, \eta_t), && t \in [T,T+K]
\end{aligned}
\ee
for a Brownian motion $(W_t)_{t\in[0,T]}$, continuous functions $\delta, \zeta : [0,T] \to \mathbb{R}_+$, and 
stochastic processes $(\xi_t)_{t\in[T,T+K]}$, $(\eta_t)_{t\in[T,T+K]}$. Equations of the form \eqref{anti} were introduced by 
Peng and Yang (2009) as duals of time-delayed forward SDEs. Their existence and uniqueness result, 
Theorem 4.2, as well as extensions for equations with jumps, can easily be derived from Proposition \ref{prop:anti}.
\end{Remark}

As an immediate consequence of Proposition \ref{prop:anti} one obtains the following result
for BSDEs with functional drivers depending on $Y_s$, $Z^M_s$ and $U^M_s$.

\begin{corollary} \label{cor:BSDEBuck}
The BSDE 
\be \label{bsdeBuck}
Y_t=\xi+\int_t^T f(s,Y_s,Z^M_s,U^M_s)ds + M_T-M_t 
\ee
has a unique solution $(Y,M) \in \bbS^2 \times \bbM^2_0$ for every terminal condition 
$\xi \in L^2(\scF_T)^d$ and driver
$$
f: [0,T] \times \Omega\times L^2(\scF_T)^d\times L^2(\scF_T)^{d\times n}
\times L^2(\Omega\times E,\scF_T \otimes {\cal B}(E), \bbP\otimes\mu;\bbR^d) \to \bbR^d
$$
satisfying the following two conditions:
\begin{itemize}
\item[{\rm (i)}]
For all $(Y,Z,U) \in \mathbb{S}^2 \times \mathbb{H}^2\times L^2(\tilde{N})$, 
$f(t,Y_t, Z_t,U_t)$ is progressively measurable with \newline
$\int_0^T \norm{f(t,0,0,0}_2 dt < \infty$.
\item[{\rm (ii)}] There exists a constant $C \ge 0$ such that
$$
\norm{f(t,Y_t,Z_t, U_t)-f(t,Y'_t,Z'_t,U'_t)}_2
\le C \brak{\norm{Y_t -Y'_t}_2+\norm{Z_t -Z'_t}_2 
+ \norm{U_t-U'_t}_{L^2(\bbP \times \mu)}}
$$
for all $t \in[0,T]$ and $(Y,Z,U),(Y',Z',U')\in \bbS^2\times\bbH^2\times L^2(\tilde{N})$.
\end{itemize} 
\end{corollary}

Corollary \ref{cor:BSDEBuck} can be used in conjunction with Theorem 2.3 to 
deduce that the following time-delayed BSDE has a unique solution. This extends Theorem 
2.3 of Delong and Imkeller (2010a) to the case of multidimensional 
BSDEs with jumps and functional dependence in the driver. In addition, our
integrability condition on the terminal condition is a bit weaker.

\begin{proposition}\label{prop:fullpath}
Let $\xi \in L^{2}(\scF_{T})^d$ and $\nu$ be a finite Borel measure on $[0,T]$. Then the BSDE
\be \label{pathBSDE}
Y_t = \xi + \int_t^T \int_{[0,s]} g(s-r, Z^M_{s-r}, U^M_{s-r}) \nu(dr) ds + M_T - M_t
\ee
has a unique solution $(Y,M) \in\bbS^{2}\times\bbM^{2}_{0}$ for every mapping
$$
g : [0,T] \times \Omega \times L^{2}(\scF_{T})^{d\times n}
\times L^{2}(\Omega\times E,\scF_{T}\otimes\scB(E),\bbP\otimes\mu;\bbR^{d})\to\bbR^d
$$ 
satisfying the following two conditions:
\begin{itemize}
\item[{\rm (i)}]
For all $(Z,U) \in \mathbb{H}^2\times L^2(\tilde{N})$, $g(t, Z_t,U_t)$ is progressively 
measurable, and $\int_0^T \norm{g(t,0,0}_2 dt < \infty$.
\item[{\rm (ii)}] There exists a constant $C \ge 0$ such that
\[
\norm{g(t,Z_{t},U_{t})-g(t,Z'_{t},U'_{t})}_{2}\leq
C\brak{\norm{Z_{t}-Z'_{t}}_{2}+\norm{U_{t}-U'_{t}}_{L^{2}(\bbP\otimes\mu)}}
\]
for all t $\in[0,T]$ and $(Z,U), (Z',U') \in \bbH^{2}\times L^{2}(\tilde N)$.
\end{itemize}
\end{proposition}

\begin{proof}
The generator corresponding to the BSDE \eqref{pathBSDE} is given by
$$
F_t(M) = \int_0^t \int_{[0,s]} g(s-r, Z^M_{s-r}, U^M_{s-r}) \nu(dr) ds.
$$
Since it does not depend on $Y$, it satisfies condition (S). So,
by Theorem \ref{thm:fix}, it is enough to show that there exists
a unique $V \in L^{2}(\scF_{T})^{d}$ such that
\be \label{fixV}
V = G(V) = \xi+ \int_0^T \int_{[0,s]} g(s-r,Z_{s-r}^{M^V},U^{M^V}_{s-r})\nu(dr)ds.
\ee
From Fubini's theorem and a change of variable, one obtains 
$$
\int_0^T \int_{[0,s]} g(s-r,Z_{s-r}^{M^V},U^{M^V}_{s-r}) \nu(dr) ds = 
\int_0^T \nu([0,T-s]) g(s,Z^{M^{V}}_s,U^{M^{V}}_s) ds.
$$
Since the driver $h(s,Z_s,U_s) = \nu([0,T-s]) g(s,Z_s,U_s)$ satisfies the conditions 
of Corollary \ref{cor:BSDEBuck}, the BSDE
$$
Y_{t} = \xi + \int_{t}^{T}h(s,Z^{M}_s,U^{M}_s)ds+M_{T}-M_{t}
$$
has a unique solution in $\bbS^{2} \times \bbM^{2}_{0}$. The associated generator, 
$\tilde{F}_t(M) = \int_0^T h(s,Z^{M}_s,U^{M}_s)ds$, does not depend on $Y$ either.
So it also satisfies condition (S), and one obtains from Theorem \ref{thm:fix} that there exists a unique 
$V\in L^{2}(\scF_{T})^{d}$ satisfying \eqref{fixV}. This completes the proof.
\end{proof}

As special cases of Corollary \ref{cor:BSDEBuck} and Proposition 
\ref{prop:fullpath}, one obtains existence and uniqueness 
results for McKean--Vlasov type BSDEs with drivers depending on the realizations
$Y_s(\omega)$, $Z^M_s(\omega)$, $U^M_s(\omega)$ as well as the distributions 
${\cal L}(Y_s), {\cal L}(Z^M_s)$, ${\cal L}(U^M_s)$ of $Y_s$, $Z^M_s$ and $U^M_s$. 
We recall that if ${\cal M}({\cal X})$ is the set of all probability measures defined on 
the Borel $\sigma$-algebra of a normed vector space 
$({\cal X},\norm{\cdot})$, the $p$-Wasserstein metric on
${\cal M}_p({\cal X}) := \crl{\eta \in {\cal M}({\cal X}) : \int_{\cal X} \|x\|^p \eta(dx) <\infty}$
is given by
$$
{\cal W}_p(\eta, \eta') := \inf \crl{\int_{{\cal X} \times {\cal X}} \|x-x'\|^p \psi(dx, dx') : \psi \in {\cal M}_p({\cal X} \times {\cal X}) 
\text{ with marginals $\eta$ and $\eta'$}}^{1/p}.
$$

The following is a consequence of Corollary \ref{cor:BSDEBuck} and
generalizes the existence and uniqueness result for mean-field 
BSDEs of Buckdahn et al. (2009).

\begin{corollary} \label{cor:MKV}
Consider a BSDE of the form
\be \label{bsde:MKV}
Y_t=\xi+\int_t^Tf(s,Y_s,Z^M_s,U^M_s, {\cal L}(Y_s), {\cal L}(Z^M_s), {\cal L}(U^M_s))ds +M_T-M_t
\ee
for a terminal condition $\xi \in L^2(\scF_T)^d$ and a driver $
f$ from $[0,T] \times\Omega\times\bbR^d\times\bbR^{d\times
	n}\times L^2(E,{\cal B}(E),\mu;\bbR^d)\times
\scM_2(\mathbb{R}^d) \times\scM_2(\mathbb{R}^{d\times n})\times
\scM_2(L^2(E,{\cal B}(E),\mu;\bbR^d))$ to $\bbR^d$.
Then \eqref{bsde:MKV} has a unique solution $(Y,M)$ in $\bbS^2\times\bbM^2_0$ if for fixed 
$$
(y,z,u,\eta,\zeta,\kappa) \in\bbR^d\times\bbR^{d\times n}\times L^2(E,{\cal B}(E),\mu;\bbR^d)\times
\scM_2(\mathbb{R}^d) \times \scM_2(\mathbb{R}^{d\times n}) \times\scM_2(L^2(E,{\cal B}(E),\mu;\bbR^d)),$$
$f(\cdot,y,z,u,\eta,\zeta, \kappa)$ is progressively measurable, and the following 
two conditions hold:
\begin{itemize}
\item[{\rm (i)}]
$\int_0^T \norm{f(t,0,0,0,{\cal L}(0),{\cal L}(0)),\scL(0)}_2 dt < \infty$
\item[{\rm (ii)}]
There exists a constant $C \ge 0$ such that
\begin{align*}
&|f(t,y,z,u,\eta,\zeta,\kappa)-f(t,y',z',u',\eta',\zeta',\kappa')|\\
&\leq 
C\brak{|y-y'|+|z-z'|+ \norm{u-u'}_{L^2(\mu)}
+\scW_2(\eta,\eta')+\scW_2(\zeta,\zeta')+\scW_2(\kappa,\kappa')}. 
\end{align*}
\end{itemize}
\end{corollary}

\begin{proof}
It follows from the assumptions that the driver $f$ is progressively measurable in 
$(t,\omega)$ and continuous in $(y,z,u,\eta,\zeta, \kappa)$. Since
$$\bbR^d\times\bbR^{d\times n}\times L^2(E,{\cal B}(E),\mu;\bbR^d)\times
\scM_2(\mathbb{R}^d) \times \scM_2(\mathbb{R}^{d\times n})
\times\scM_2(L^2(E,{\cal B}(E),\mu;\bbR^d))$$ is a separable metric
space, one obtains from Lemma 4.51 of Aliprantis and Border (2006) that $f$ is jointly measurable 
in all its arguments. This implies that $f(t,Y_t,Z_t,U_t,{\cal L}(Y_t), {\cal L}(Z_t),\scL(U_t))$
is progessively measurable for every triple $(Y,Z,U) \in \bbS^2 \times \bbH^2 \times U \in L^2(\tilde{N})$.
It follows that condition (i) of Corollary \ref{cor:BSDEBuck} holds, and it just remains to show that
\begin{align*}
&\norm{f(t,Y_t,Z_t,U_t,{\cal L}(Y_t), {\cal L}(Z_t),\scL(U_t)) 
- f(t,Y'_t,Z'_t,U'_t,{\cal L}(Y'_t), {\cal L}(Z'_t),\scL(U'_t)}_2\\
&\leq D
\brak{\norm{Y_t-Y'_t}_2+\norm{Z_t -Z'_t}_2+ \norm{U_t -U'_t}_{L^2(\bbP \times \mu)}}
\end{align*}
for some constant $D$. But this is a consequence of condition (ii) since one has
$$
{\cal W}^2_2({\cal L}(Y_t), {\cal L}(Y'_t)) 
\le \int_{\mathbb{R}^d \times \mathbb{R}^d} |y-y'|^2 {\cal L}(Y_t,Y'_t)(dy,dy')
= \norm{Y_t - Y'_t}^2_2,
$$
and analogously, 
$$
{\cal W}^2_2({\cal L}(Z_t), {\cal L}(Z'_t)) \leq\norm{Z_t - Z'_t}^2_2, \quad
 \scW^2_2(\scL(U_t),\scL(U_t)) \leq \norm{U_t-U'_t}^2_{L^2(\bbP \times \mu)}.
$$
\end{proof}

Using the same arguments as in the proof of Corollary \ref{cor:MKV}, one obtains 
from Proposition \ref{prop:fullpath} the following result for time-delayed
McKean--Vlasov type BSDEs.

\begin{corollary} \label{cor:MKVpath}
Consider a BSDE of the form
\be \label{bsde:MKVpath}
Y_t= \xi + \int_t^T \int_0^s g(s-r, Z^M_{s-r} ,U^M_{s-r}, {\cal L}(Z^M_{s-r}), {\cal L}(U^M_{s-r}))
\nu(dr) ds + M_T-M_t
\ee
for a terminal condition $\xi \in L^2(\scF_T)^d$, a finite Borel measure $\nu$ on $[0,T]$
and a mapping
$$
g : [0,T] \times \Omega \times\bbR^{d\times n}\times L^2(E,{\cal B}(E),\mu;\bbR^d)\times
\scM_2(\mathbb{R}^{d \times n}) \times\scM_2(L^2(E,{\cal B}(E),\mu;\bbR^d))\to\bbR^d.
$$
Then \eqref{bsde:MKVpath} has a unique solution $(Y,M)$ in $\bbS^2\times\bbM^2_0$ if for fixed 
$$
(z,u,\zeta,\kappa) \in \bbR^{d\times n} \times L^2(E,{\cal B}(E),\mu;\bbR^d)\times
\scM_2(\mathbb{R}^{d\times n}) \times\scM_2(L^2(E,{\cal B}(E),\mu;\bbR^d)),$$
$g(\cdot,z,u,\zeta,\kappa)$ is progressively measurable, and the following 
two conditions hold:
\begin{itemize}
\item[{\rm (i)}]
$\int_0^T \norm{g(t,0,0,{\cal L}(0)),\scL(0)}_2 dt < \infty$
\item[{\rm (ii)}]
There exists a constant $C \ge 0$ such that
$$
|g(t,z,u,\zeta,\kappa)-g(t,z',u',\zeta',\kappa')| \leq 
C\brak{|z-z'|+ \norm{u-u'}_{L^2(\mu)} +\scW_2(\zeta,\zeta')+\scW_2(\kappa,\kappa')}. 
$$
\end{itemize}
\end{corollary}

\setcounter{equation}{0}
\section{Existence of solutions to non-Lipschitz equations}
\label{sec:compact}

In this section we use compactness assumptions to derive
existence results for different BSEs and BSDEs with non-Lipschitz coefficients. 
To find compact sets in the space $L^2({\cal F}_T)^d$, 
we assume in all of Section \ref{sec:compact}
that the sample space $\Omega$ is an infinite-dimensional separable Hilbert space with
inner product $\ang{\cdot,\cdot}$ and corresponding norm $\norm{\cdot}$.
We fix a complete orthonormal system $e_j$, $j \in \bbN$, of $\Omega$ together with positive 
numbers $\lambda_j$, $j \in \bbN$ satisfying $\sum_{j \in \bbN} \lambda_j < \infty$. Then
$Qe_j := \lambda_j e_j$ defines a positive self-adjoint trace class operator $Q : \Omega \to \Omega$.
The mean zero Gaussian measure $\bbP$ with covariance $Q$ is the unique
probability measure on the Borel $\sigma$-algebra ${\cal B}(\Omega)$ of $\Omega$ under which the 
functions $\phi_j(\omega) = \ang{\omega,e_j}$, $j \in \mathbb{N}$, are
independent normal random variables with mean zero and variance $\lambda_j$, $j \in \mathbb{N}$;
see Da Prato (2006) for details. 
The map $e_j \mapsto \phi_j/\sqrt{\lambda_j}$ has a unique 
continuous linear extension $W : \Omega \to L^2(\Omega)$, called white noise mapping. It 
is an isometry between $\Omega$ and the closed subspace of $L^2(\Omega)$ generated by
$\phi_j$, $j \in \bbN$. 

To define the Sobolev space $W^{1,2}(\Omega)$ in $L^2(\Omega)$, let $\scE(\Omega)$ 
be the linear span of all real and imaginary parts of functions of the form $\omega \mapsto e^{i\ang{\omega,\eta}}$
for some $\eta \in \Omega$. For $\varphi \in \scE(\Omega)$, we denote by $D_j \varphi$ 
the derivative of $\varphi$ in the direction of $e_j$:
$$
D_j \varphi(\omega)= \lim_{\varepsilon \to 0} \frac{\varphi(\omega + \varepsilon e_j) - \varphi(\omega)}{\varepsilon}.
$$
The mapping $D : \scE(\Omega) \subseteq L^2(\Omega) \to L^2(\Omega;\Omega)$, 
$\varphi \mapsto D\varphi := \sum_{j \in \bbN} D_j \varphi e_j$
is closable. We maintain the notation $D$ for the closure of $D$ and denote its domain by
$W^{1,2}(\Omega)$. Endowed with the inner product
$$
\ang{\varphi,\psi}_{W^{1,2}} := \bbE\brak{\varphi\psi+\ang{D\varphi,D\psi}},
$$
the Sobolev space $W^{1,2}(\Omega)$ becomes a Hilbert space.
For $\varphi \in L^2(\Omega)^d$ and $\psi \in W^{1,2}(\Omega)^d$, we set
$$
\norm{\varphi}_2^2 :=\sum_{i=1}^d \bbE \varphi_i^2, \quad 
\norm{D\psi}^2_2 := \sum_{i=1}^d\bbE \ang{D\psi_i,D\psi_i} \quad 
\mbox{and} \quad \norm{\psi}_{W^{1,2}}^2
:= \norm{\psi}^2_2 + \norm{D \psi}^2_2.
$$
Theorem 10.25 of Da Prato (2006) shows that every $\varphi \in W^{1,2}(\Omega)^d$
satisfies the Poincar\'e inequality 
\be \label{Poinc}
\bbE |\varphi -\mathbb{E} \varphi|^2 \le \lambda \norm{D\varphi}^2_2 \quad
\mbox{for } \lambda := \max_j \lambda_j.
\ee
Moreover, by Theorem 10.16 of Da Prato (2006), every bounded set in $W^{1,2}(\Omega)^d$ is 
relatively compact in $L^2(\Omega)^d$. 

We say a function $\varphi : \Omega \to \mathbb{R}^d$ is $\omega$-Lipschitz with constant 
$L \ge 0$ if 
$$\abs{\varphi(\omega)-\varphi(\omega')} \le L \norm{\omega-\omega'} \quad \mbox{for all } 
\omega, \omega' \in \Omega.
$$
It follows from Propositon 10.11 of Da Prato (2006) that every $\omega$-Lipschitz 
function $\varphi : \Omega \to \mathbb{R}^d$ with constant $L$ belongs to $W^{1,2}(\Omega)^d$ 
with $\norm{D\varphi}_2 \le L$. In particular, one obtains that for given numbers $K,L \ge 0$,
the set of all $\omega$-Lipschitz $\varphi : \Omega \to \mathbb{R}^d$ with constant $L$ 
satisfying $|\mathbb{E}\varphi| \le K$ is compact in $L^2(\Omega)^d$. Moreover, the following holds:

\begin{lemma} \label{lemma:Lip}
Let $h : l^1 \to \mathbb{R}^d$ be a mapping satisfying 
$|h(x) - h(y)| \le K \norm{x-y}_1$ for some constant $K \ge 0$. 
Then for any $x \in l^2$, 
$$
\varphi = h \brak{\sqrt{\lambda_j} x_j W(e_j), j \in \bbN}
$$ 
is an $\omega$-Lipschitz random variable with constant $K \norm{x}_2$.
\end{lemma}

\begin{proof}
One has
$|\varphi(\omega) - \varphi(\omega')|  \le K \norm{x_j \ang{\omega-\omega',e_j}, j \in \bbN}_1
\le  K \norm{x}_2 \norm{\omega-\omega'}$.
\end{proof}

\begin{Remark} \label{rem:BMPM}
The assumptions on $\Omega$ in this section are not restrictive for the purpose of studying BSEs and BSDEs.
For instance, they allow for probability spaces rich enough to support an
$n$-dimensional Brownian motion together with an independent Poisson random measure
on $[0,T] \times \bbR^m \setminus \crl{0}$. For an explicit construction, one can e.g., choose 
$\Omega$ to be of the form $\Omega = L^2([0,T];\bbR^n) \oplus l^2$,
where $L^2([0,T];\bbR^{n})$ is the space of square-integrable measurable functions from $[0,T]$ to $\bbR^n$ 
and $l^2$ the space of square-summable sequences. The inner product on 
$L^2([0,T];\bbR^n) \oplus l^2$ is given by
$$
\ang{(h,x), (h',x')} = \int_0^T h(s) \cdot h'(s) ds 
+ \sum_{j \in \bbN} x_j x'_j,
$$
where $\cdot$ denotes the standard scalar product on $\bbR^n$. Let 
$\mathbb{P}$ be a mean zero Gaussian measure corresponding to  
a positive self-adjoint trace class operator given by $Q e_j = \lambda_j e_j$
for a complete orthonormal system $(e_j)$ of $\Omega$ and positive numbers $(\lambda_j)$ 
satisfying $\sum_{j \in \bbN} \lambda_j < \infty$. If $W : \Omega \to L^2(\Omega)$ 
is the corresponding white noise mapping,
$b_i$ denotes the $i$-th unit vector in $\bbR^n$ and $(c_j)$ is a complete orthonormal system in $l^2$,
then $W^i_t := W(b_i 1_{[0,t]},0)$ defines an $n$-dimensional Brownian motion
independent of the sequence $\zeta_j := W(0,c_j)$ of independent standard normals.
For a given $\sigma$-finite measure $\mu$ on the Borel $\sigma$-algebra of
$\mathbb{R}^m \setminus \crl{0}$,
a Poisson random measure $N$ on $[0,T] \times \mathbb{R}^m \setminus \crl{0}$ with 
intensity measure $dt \mu(dx)$ can be realized as a function of $\zeta_j$, $j \in \bbN$.
Alternatively, $N$ can be realized with only $\zeta_{2j-1}$, $j \in \bbN$, and
$\zeta_{2j}$, $j \in \bbN$, can be used to model additional noise.
\end{Remark}

\subsection{Non-Lipschitz BSEs and BSDEs with path-dependent generators}
\label{subsec:NLBSE}

Denote by ${\cal F}$ the completion of the Borel $\sigma$-algebra 
${\cal B}(\Omega)$ with respect to $\bbP$, and let 
$\bbF = ({\cal F}_t)_{t \in [0,T]}$ be a general filtration satisfying the usual conditions.
The following theorem provides a general existence result for non-Lipschitz BSEs.
It uses the theorem of Krasnoselskii (1964), which combines the fixed point 
results of Banach and Schauder; for a textbook treatment, see e.g., Smart (1974).

\begin{theorem} \label{thm:contrcont}
Let $\xi \in L^2({\cal F}_T)^d$ and assume $F$ is of the form $F = F^1 + F^2$ for mappings 
$F^1, F^2 : \bbS^2 \times \bbM^2_0 \to \bbS^2_0$. Then the BSE \eqref{bse} has a 
solution $(Y,M) \in\bbS^2\times\bbM^2_0$ if there exist constants 
$C < 1$ and $R_1,R_2,R_3 \ge 0$ such that the following hold:

\begin{itemize}
\item[{\rm (i)}] 
$\norm{F(Y,M)-F(Y',M)}_{\bbS^2} \le C \norm{Y-Y'}_{\bbS^2}$
and $F(Y,M) \in \bbS^2_0$ is continuous in $M \in \bbM^2_0$
\item[{\rm (ii)}] 
$\norm{F^1_T(Y,M)-F^1_T(Y',M')}_2 \le C
\sqrt{\norm{Y_0-Y'_0}^2_2 + \norm{M -M'}^2_{\bbS^2}/4}$
\item[{\rm (iii)}]
For all $(Y,M) \in \bbS^2 \times \bbM^2_0$ satisfying
$\sqrt{\norm{Y_0}^2_2 + \norm{M}^2_{\bbS^2}/4} \le R_1$, one has
$F^2_T(Y,M) \in  W^{1,2}(\Omega )^d$ with $\norm{F^2_T(Y,M)}_2 \le R_2$ and 
$\norm{D F^2_T(Y,M)}_2 \le R_3$
\item[{\rm (iv)}]
$\norm{\xi}_2 + \norm{F_T^1(0,0)}_{2} + C R_1 + R_2 \le R_1$.
\end{itemize}
\end{theorem}

\begin{proof}
By Lemma \ref{lemma:LipY}, it follows from condition (i) that $F$ satisfies (S).
So by Theorem \ref{thm:fix}, it is enough to show that the mapping 
$V \mapsto G(V) = \xi + F_T(Y^V,M^V)$ has a fixed point in $L^2({\cal F}_T)^d$.
To do that we define ${\cal C} :=\crl{V \in L^2({\cal F}_T)^d :\norm{V}_2 \leq R_1}$, 
$G^1(V) :=\xi+F^1_T(Y^V,M^V)$, $G^2(V) :=F^2_T(Y^{V},M^{V})$ and show the
following: 1) $G^1$ is a contraction on $L^2({\cal F}_T)^d$; 2) $G^2$ is continuous with respect
to $\|.\|_2$; 3) $G^2$ maps $\scC$ into a compact subset of $L^2(\scF_T)^d$; and 4)
$G^1(V) + G^2(V') \in \scC$ for all $V, V' \in \scC$. Then it follows from 
Krasnoselskii's theorem that $G$ has a fixed point.

Step 1: $G^1: L^2(\scF_T)^d \to  L^2(\scF_T)^d$ is a contraction with respect to $\|.\|_2$:\\
It follows from (ii) that
$$ \norm{G^1(V)-G^1(V')}^2_2 = \norm{F_T^1(Y^V,M^V)-F_T^1(Y^{V'},M^{V'})}^2_{2}
\le C^2 \brak{ \norm{Y^V_0-Y^{V'}_0}^2_2 + \frac{1}{4} \norm{M^V-M^{V'}}^2_{\bbS^2}}.
$$
By Doob's $L^2$-maximal inequality, one has $\norm{M^V-M^{V'}}_{\bbS^2} \le 2 \norm{M_T^V-M_T^{V'}}_2$.
Therefore,
$$
\norm{G^1(V)-G^1(V')}^2_2 \le
C^2 \brak{\norm{\bbE_0(V-V')}^2_2 + \norm{M^V_T-M_T^{V'}}^2_2}
\leq C^2 \norm{V-V'}^2_2, 
$$
which shows that $G^1$ is a contraction.

Step 2: $G^2 :  L^2(\scF_T)^d \to  L^2(\scF_T)^d$ is continuous with respect to $\|.\|_2$:\\
By Doob's $L^2$-maximal inequality, $V \mapsto M^V$ is a continuous mapping from $L^2({\cal F}_T)^d$ 
to $\bbM^2_0$. Moreover, since 
$$
Y^V_t = \hat{M}^V_t - F_t(Y^V,M^V) \quad \mbox{for}\quad  \hat{M}^V_t := \bbE_t V = \bbE_0 V - M^V_t,
$$
one obtains from the first part of condition (i) that
\beas
&& \norm{Y^V - Y^{V'}}_{\bbS^2} \le \norm{\hat{M}^V - \hat{M}^{V'}}_{\bbS^2}
+ \norm{F(Y^V, M^V) - F(Y^{V'}, M^{V'})}_{\bbS^2}\\
&\le& 2 \norm{V - V'}_2
+ \norm{F(Y^V, M^V) - F(Y^V, M^{V'})}_{\bbS^2} 
+ \norm{F(Y^V, M^{V'}) - F(Y^{V'}, M^{V'})}_{\bbS^2}\\
&\le& 2\norm{V - V'}_2 + \norm{F(Y^V, M^V) - F(Y^V, M^{V'})}_{\bbS^2} + C \norm{Y^V- Y^{V'}}_{\bbS^2}.
\eeas
Therefore,
$$
(1-C) \norm{Y^V - Y^{V'}}_{\bbS^2} \le 
2 \norm{V - V'}_2
+ \norm{F(Y^V, M^V) - F(Y^V, M^{V'})}_{\bbS^2},
$$
and it follows from the second part of (i) 
that $V \mapsto Y^V$ is continuous from $L^2({\cal F}_T)^d$ to $\bbS^2$.
Since $F^2 = F - F^1$, one obtains from (i) and (ii) 
that $(Y^V, M^V) \mapsto F^2_T(Y^V, M^V)$ is continuous from $\bbS^2 \times \bbM^2_0$
to $L^2({\cal F}_T)^d$. This proves the continuity of $G^2$.

Step 3: $G^2(\scC)$ is contained in a compact subset of $L^2(\scF_T)^d$:\\
For $V \in {\cal C}$, one has 
\be \label{Vest}
\norm{Y^V_0}^2_2 + \frac{1}{4} \norm{M^V}^2_{\bbS^2}
\le \norm{\bbE_0 V}^2_2 + \norm{M^V_T}^2_2 = \norm{V}^2_2 \le R^2_1.
\ee
So it follows from (iii) that $F^2_T(Y^V,M^V)$ is in $W^{1,2}(\Omega)^d$ with $\norm{F^2_T(Y^V,M^V)}_2 \le R_2$ and 
\linebreak $\norm{D F^2_T(Y^V,M^V)}_2 \le R_3$. Since bounded subsets of 
$W^{1,2}(\Omega)^d$ are relatively compact in $L^2(\Omega)^d$, this shows that 
$G^2(\scC)$ is contained in a compact subset of $L^2(\scF_T)^d$.

Step 4: $G^1(V)+G^2(V')\in\scC$ for all $V,V'\in\scC$:\\
If $V \in {\cal C}$, one obtains from \eqref{Vest} that 
$\norm{Y^V_0}^2_2 + \norm{M^V}^2_{\bbS^2}/4 \le R^2_1$.
So it follows from (ii) that
\beas
\norm{G^1(V)}_2 &\le& \norm{\xi}_2+\norm{F^1_T(Y^V,M^V)}_{2}
\leq\norm{\xi}_2+\norm{F^1_T(0,0)}_{2}+ C (\norm{Y_0^V}^2_2 + \norm{M^V}^2_{\bbS^2}/4)^{1/2}\\
&\le& \norm{\xi}_2 + \norm{F_T^1(0,0)}_{2} + C R_1.
\eeas
By (iii), one has $\norm{G^2(V')}_2 \le R_2$. Therefore, one obtains from (iv) that
$\norm{G^1(V) + G^2(V')}_2 \le R_1$.

So Krasnoselskii's theorem applies, and one can conclude 
that $G$ has a fixed point in $L^2({\cal F}_T)^d$.
\end{proof}

Assumption (i) of Theorem \ref{thm:contrcont} is needed to ensure that condition (S) holds and
$F^2_T(Y, M)$ is continuous in $(Y, M)$. In the following special case it is not needed.

\begin{proposition} \label{prop:contrcont} 
Let $\xi \in L^2({\cal F}_T)^d$ and assume $F$ is of the form 
$F(Y,M) = F^1(Y_0,M) + F^2(Y_0,M)$ for mappings 
$F^1, F^2 : L^2({\cal F}_0)^d \times \bbM^2_0 \to \bbS^2_0$. 
Then the BSE \eqref{bse} has a solution $(Y,M) \in \bbS^2\times\bbM^2_0$ if there 
exist a constant $C <1$ and a nondecreasing function $\rho : \mathbb{R}_+ \to \mathbb{R}_+$ 
satisfying
\be \label{1-C}
\limsup_{x \to \infty} \frac{\rho(x)}{x} < 1-C
\ee
such that the following two conditions hold:
\begin{itemize}
\item[{\rm (i)}]
$\norm{F^1_T(Y_0,M)-F^1_T(Y'_0,M')}_2 \le 
C \sqrt{\norm{Y_0-Y'_0}^2_2 + \norm{M -M'}^2_{\bbS^2}/4}$
\item[{\rm (ii)}] 
$F^2_T : L^2(\scF_0)^d \times \bbM^2_0 \to L^2(\scF_T)^d$ is continuous and takes values
in $W^{1,2}(\Omega)^d$ with 
$$
|\mathbb{E} F^2_T(Y_0,M)|^2 + \lambda \norm{D F^2_T(Y_0,M)}^2_2
\le \rho^2 \brak{\sqrt{\norm{Y_0}^2_2 + \norm{M}^2_{\bbS^2}/4}}.
$$
\end{itemize}
\end{proposition}

\begin{proof}
Since $F$ only depends on $Y_0$ and $M$, condition (S) holds trivially.
By Theorem \ref{thm:fix}, the proposition follows if we can show that
$V \mapsto G(V) = \xi + F_T(Y_0^V,M^V)$ has a fixed point in $L^2({\cal F}_T)^d$.
To do that, we fix a constant $R_1 \ge 0$ and define ${\cal C}$, $G^1$ and $G^2$ as in the proof of 
Theorem \ref{thm:contrcont}. Then one obtains from (i) like in the proof of Theorem \ref{thm:contrcont}
that $G^1$ is a contraction on $L^2({\cal F}_T)^d$. Condition (ii) implies that $G^2$ is continuous with 
respect to $\|.\|_2$, and since 
$$
\rho^2 \brak{\sqrt{\norm{Y^V_0}^2_2 + \norm{M^V}^2_{\bbS^2}/4}}
\le \rho^2 \brak{\sqrt{\norm{Y^V_0}^2_2 + \norm{M^V_T}^2_2}} = \rho^2 (\norm{V}_2),
$$
that $G^2({\cal C})$ is relatively compact in $L^2({\cal F}_T)^d$. Due to \eqref{1-C}, one has
$$
\norm{\xi}_2 + \norm{F_T^1(0,0)}_{2} + C R_1 + \rho(R_1) \le R_1
$$
if $R_1$ is chosen large enough. Then for $V,V' \in {\cal C}$, 
\beas
\norm{G^1(V)}_2 &\le& \norm{\xi}_2 + \norm{F^1_T(Y_0^V, M^V)}_2
\le \norm{\xi}_2 + \norm{F^1_T(0,0)}_2 + C(\norm{Y_0^V}^2_2 + \norm{M^V_T}^2_2)^{1/2}\\
&\le& \norm{\xi}_2 + \norm{F^1_T(0,0)}_2 + CR_1,
\eeas
and, by Poincar\'e's inequality,
\beas
\norm{G^2(V')}^2_2 &\le& |\mathbb{E} F^2_T(Y_0^{V'},M^{V'})|^2 + \lambda 
\norm{DF^2_T(Y_0^{V'},M^{V'})}_2^2 \le
\rho^2 \brak{\sqrt{\|Y_0^{V'}\|^2_2 + \|M^{V'}\|^2_{\bbS^2}/4}}\\
&\le& \rho^2 \brak{\sqrt{\|Y_0^{V'}\|^2_2 + \|M^{V'}_T\|^2_2}} = \rho^2(\norm{V'}_2).
\eeas
Therefore,
$$
\norm{G^1(V) + G^2(V')}_2 \le \norm{\xi}_2 + \norm{F^1_T(0,0)}_2 + CR_1
+ \rho(R_1) \le R_1,
$$
and it follows from Krasnoselskii's theorem that $G$ has a fixed point in 
$L^2({\cal F}_T)^d$.
\end{proof}

As a consequence of Proposition \ref{prop:contrcont} one obtains an existence result for BSDEs 
\be \label{genbsde}
Y_t = \xi + \int_t^T f(s,Y_0,M) ds + M_T - M_t
\ee
with drivers $f$ depending on $Y_0$ and the whole martingale $M$.

\begin{corollary} \label{cor:YM}
Let $\xi \in L^2({\cal F}_T)^d$ and assume $f$ to be of the form
$f = f^1 + f^2$ for mappings $f^1,f^2 : [0,T] \times \Omega \times L^2(\scF_0)^d\times \bbM^2_0\to\bbR^d$.
Then the BSDE \eqref{genbsde} has a solution $(Y,M) \in \bbS^2 \times \bbM^2_0$ 
if there exist a constant $C < T^{-1}$ and a nondecreasing function 
$\rho : \mathbb{R}_+ \to \mathbb{R}_+$ satisfying
$$
\limsup_{x \to \infty} \frac{\rho(x)}{x} < 1-C T
$$
such that the following two conditions hold:
\begin{itemize}
\item[{\rm (i)}]
For all $(Y_0, M) \in L^2(\scF_0)^d \times \bbM^2_0$, $f^1(\cdot,Y_0,M)$ is 
progressively measurable with $\int_0^T |f^1(t,0,0) | dt \in L^2({\cal F}_T)$, and 
$$
\norm{f^1(t,Y_0,M)-f^1(t,Y'_0,M')}_2 \le 
C \sqrt{\norm{Y_0-Y'_0}^2_2 + \norm{M -M'}^2_{\bbS^2}/4}.
$$
\item[{\rm (ii)}] 
For all $(Y_0, M) \in L^2(\scF_0)^d \times \bbM^2_0$, $f^2(.,Y_0,M)$ is 
progressively measurable with $\int_0^T|f^2(t,Y_0,M)|dt \in L^2({\cal F}_T)$, and 
$J(Y_0,M) := \int_0^T f^2(t,Y_0,M) dt$ defines a continuous mapping 
$J : L^2({\cal F}_0)^d \times \bbM^2_0 \to L^2({\cal F}_T)^d$ with values in 
$W^{1,2}(\Omega)^d$ such that 
$$
|\mathbb{E} J(Y_0,M)|^2 + \lambda \norm{DJ(Y_0,M)}^2_2 
\le \rho^2 \brak{\sqrt{\norm{Y_0}^2_2 + \norm{M_T}^2_{\bbS^2}/4}}.
$$
\end{itemize}
\end{corollary}

\begin{proof}
It follows from the assumptions that for all $Y_0$ and $M$, 
$F^i_t(Y_0,M) = \int_0^t f^i(s,Y_0,M) ds$ belongs to $\bbS^2_0$ for $i=1,2$, and
$$
\bbE \abs{F^1_T(Y_0,M) - F^1_T(Y',M')}^2
\le C^2T^2 \brak{\norm{Y_0 - Y'_0}_2^2 + \norm{M - M'}^2_{\bbS^2}/4}.
$$
So the conditions of Proposition \ref{prop:contrcont} hold with 
$C T$ instead of $C$, and the corollary follows.
\end{proof}

If $F$ does not depend on $Y$, the assumptions of Theorem \ref{thm:contrcont}
can be relaxed further, and one obtains the following

\begin{theorem} \label{thm:FM}
Let $\xi \in L^2({\cal F}_T)^d$ and assume $F$ is of the form $F(Y,M) = F^1(M) + F^2(M)$ 
for mappings $F^1, F^2 : \bbM^2_0 \to \bbS^2_0$. Then the BSE \eqref{bse} has a solution 
$(Y,M) \in \bbS^2 \times \bbM^2_0$ if there exist a constant $C <1/2$ and a 
nondecreasing function $\rho : \mathbb{R}_+ \to \mathbb{R}_+$ satisfying
\be \label{rhoC}
\limsup_{x \to \infty} \frac{\rho(x)}{x} < \frac{1/2-C}{\sqrt{\lambda}}
\ee
such that the following two conditions hold:
\begin{itemize}
\item[{\rm (i)}]
$\norm{F^1_T(M) - \bbE_0F^1_T(M) - (F^1_T(M') -\bbE_0F^1_T(M'))}_2 
\le C \norm{M -M'}_{\bbS^2}$
\item[{\rm (ii)}] 
$F^2_T : \bbM^2_0 \to L^2(\scF_T)^d$ is continuous and
takes values in $W^{1,2}(\Omega)^d$ with $\norm{DF^2_T (M)}_2 \le \rho(\norm{M}_{\bbS^2})$.
\end{itemize}
\end{theorem}

\begin{proof}
By Corollary \ref{cor:fix}, it is enough to show that the mapping 
$$V \mapsto G_0(V) = \xi - \bbE_0 \xi + F_T(M^V) - \bbE_0 F_T(M^V)$$
has a fixed point in $L^2_0({\cal F}_T)^d$. For a given constant $R \ge 0$,
define ${\cal C} :=\crl{V \in L^2_0({\cal F}_T)^d :\norm{V}_2 \leq R}$, 
$G^1_0(V) := \xi - \bbE_0\xi + F^1_T(M^V) - \bbE_0 F^1_T(M^V)$ and
$G^2_0(V) := F^2_T(M^{V}) - \bbE_0 F^2_T(M^{V})$. By (i) and Doob's $L^2$-maximal 
inequality, one has
\beas
&& \norm{G^1_0(V)-G^1_0(V')}_2 \le \norm{F^1_T(M^V) - \bbE_0F^1_T(M^V) - (F^1_T(M^{V'}) -\bbE_0F^1_T(M^{V'}))}_2\\
&\le& C \norm{ M^V-M^{V'}}_{\bbS^2} \le 2C \norm{M^V_T-M^{V'}_T}_2 \le 2C \norm{V-V'}_2.
\eeas
So $G^1_0$ is a contraction on $L^2_0({\cal F}_T)^d$. Moreover, it follows from (ii) that 
$G^2_0: L^2_0(\scF_T)^d \to  L^2_0(\scF_T)^d$ is continuous and 
$G^2_0({\cal C})$ is relatively compact in $L^2_0(\scF_T)^d$. 
Finally, let $V,V' \in {\cal C}$. Then 
$$
\norm{G^1_0(V)}_2 \le \norm{\xi - \bbE_0 \xi}_2 + \norm{F^1_T(0) - \bbE_0 F^1_T(0)}_2 + 2 C R,
$$
and 
$$
\norm{G^2_0(V')}_2 = \norm{F^2_T(M^{V'}) - \bbE_0 F^2_T(M^{V'})}_2
\le \sqrt{\lambda} \norm{D F^2_T(M^{V'})}_2 \le \sqrt{\lambda} \rho \brak{\norm{M^{V'}}_{\bbS^2}}
\le \sqrt{\lambda} \rho(2 R).
$$
By \eqref{rhoC}, one has $G^1_0(V) + G^2_0(V') \in {\cal C}$ for $R$ large enough.
So it follows like in the proof of Theorem \ref{thm:contrcont}
from Krasnoselskii's theorem that $G_0 = G^1_0+G^2_0$ has a 
fixed point in $L^2_0({\cal F}_T)^d$.
\end{proof}

\begin{corollary} \label{corcor:Lip}
A BSDE of the form
$$
Y_t = \xi + \int_t^T (f^1(s,M) + f^2(s,M)) ds + M_T - M_t
$$
for a terminal condition $\xi \in L^2({\cal F}_T)^d$ and mappings 
$f^1,f^2 : [0,T] \times \Omega \times \bbM^2_0\to\bbR^d$ 
has a solution $(Y,M) \in \bbS^2 \times \bbM^2_0$ 
if  there exist a constant $C < (2T)^{-1}$ and a nondecreasing function 
$\rho : \mathbb{R}_+ \to \mathbb{R}_+$ satisfying
$$
\limsup_{x \to \infty} \frac{\rho(x)}{x} < \frac{1/2-C T}{\sqrt{\lambda}}
$$
such that the following two conditions hold:
\begin{itemize}
\item[{\rm (i)}]
For all $M \in \bbM^2_0$, $f^1(.,M)$ is progressively measurable with 
$\int_0^T |f^1(t,0) | dt \in L^2({\cal F}_T)$, and 
$$
\norm{f^1(t,M)-f^1(t,M')}_2 \le C \norm{M -M'}_{\bbS^2}
$$
\item[{\rm (ii)}] 
For all $M \in \bbM^2_0$, $f^2(.,M)$ is progressively measurable with
$\int_0^T|f^2(t,M)|dt \in L^2({\cal F}_T),$ and $J(M) := \int_0^T f^2(t,M) dt$ 
defines a continuous map $J : \bbM^2_0 \to L^2({\cal F}_T)^d$ such that for all 
$M \in M^2_0$, $J(M)$ is $\omega$-Lipschitz with constant $\rho(\norm{M}_{\bbS^2})$.
\end{itemize}
\end{corollary}

\begin{proof}
As in Corollary \ref{cor:YM}, it follows from the assumptions that 
$F^i_t(M) = \int_0^t f^i(s,M) ds$ is in $\bbS^2_0$ for $i=1,2$ and all $M \in \bbM^2_0$.
Moreover, 
$$
\bbE \abs{F^1_T(M) - F^1_T(M')}^2 \le C^2 T^2 \norm{M - M'}^2_{\bbS^2},
$$
and since $\int_0^T f^2(s,M)ds$ is $\omega$-Lipschitz with constant
$\rho(\norm{M}_{\bbS^2})$, one has 
$\norm{DF^2_T (M)}_2 \le \rho(\norm{M}_{\bbS^2})$. So the conditions of 
Theorem \ref{thm:FM} hold with $CT$ instead of $C$, and the corollary follows
as a consequence.
\end{proof}

\begin{Remark}
As a special case of Corollary \ref{corcor:Lip}, one obtains that the BSDE
$$
Y_t = \xi + \int_t^T f(s,M) ds + M_T - M_t
$$
has a solution for every terminal condition $\xi \in L^2({\cal F}_T)^d$ and driver 
$f$ satisfying condition (ii) of Corollary \ref{corcor:Lip}. This provides an existence 
result for multidimensional BSDEs with drivers exhibiting general dependence on the whole 
process $M$. In contrast to the BSDE results in Section \ref{sec:contr}, here the driver 
is not required to be Lipschitz in $M$. On the other hand, it is supposed to satisfy 
the $\omega$-Lipschitzness assumption contained in condition (ii) of Corollary \ref{corcor:Lip}. 
\end{Remark}

\subsection{Non-Lipschitz BSDEs based on a Brownian motion and a Poisson random measure}
\label{subsec:bsde}

We now focus on BSDEs with non-Lipschtiz coefficients that depend on an
$n$-dimensional Brownian motion $W$ and an independent Poisson random measure 
$N$ on $[0,T] \times E$, where $E = \bbR^m \setminus \crl{0}$,
with an intensity measure of the form $dt \mu(dx)$ for a measure $\mu$ over 
the Borel $\sigma$-algebra ${\cal B}(E)$ of $E$ satisfying
$$
\int_E (1 \wedge |x|^2) \mu(dx) < \infty
$$
(see Remark \ref{rem:BMPM} above for a construction of $W$ and $N$ 
in the case where $\bbP$ is a mean zero Gaussian measure on the
infinite-dimensional separable Hilbert space $\Omega$).

As in Subsection \ref{subsec:NLBSE}, we denote by ${\cal F}$ the completed Borel $\sigma$-algebra 
on $\Omega$ and let $\bbF =  ({\cal F}_t)_{0\le t\le T}$ be a filtration satisfying the usual conditions.
Let $\tilde{N}$ be the compensated random measure \linebreak $N(dt,dx) - dt \mu(dx)$, and assume
that, for $A \in {\cal B}(E)$ with $\mu(A) < \infty$, $\tilde{N}([0,t] \times A)$ and $W$ are martingales with respect to $\bbF$.
The next proposition gives an existence result for BSDEs with functional drivers of the form 
\be \label{bsde:ZU}
Y_t = \xi + \int_t^T f(s,Z^M_s,U^M_s) ds + M_T - M_t.
\ee

\begin{proposition} \label{prop:f1f2}
Let $\xi \in L^2({\cal F}_T)^d$ and assume the driver is of the form $f = f^1 + f^2$ for mappings 
$$
f^1, f^2 : [0,T] \times \Omega \times L^2({\cal F}_T)^{d \times n} \times 
L^2(\Omega \times E, {\cal F}_T \otimes {\cal B}(E), \bbP \otimes \mu)^d \to \bbR^d.
$$ 
Then the BSDE \eqref{bsde:ZU} has a solution $(Y,M) \in \bbS^2 \times \bbM^2_0$ 
if there exist a constant $C \ge 0$ and a nondecreasing function $\rho : \mathbb{R}_+ \to \mathbb{R}_+$ 
such that for all $M,M' \in \bbM^2_0$, the following two conditions hold:
\begin{itemize}
\item[{\rm (i)}]
$f^1(t,Z^M_t,U^M_t)$ is progressively measurable with
$\int_0^T \norm{f^1(t,0,0)}_2 dt < \infty$, and
$$
\norm{f^1(t,Z^M_t,U^M_t)-f^1(t,Z^{M'}_t,U^{M'}_t)}_2 \le 
C \brak{\norm{Z^M_t - Z^{M'}_t}_2 + \norm{U^M_t - U^{M'}_t}_{L^2(\bbP \otimes \mu)}}
$$
\item[{\rm (ii)}] 
$f^2(t,Z^M_t,U^M_t)$ is progressively measurable with
$\int_0^T \norm{f^2(t,0,0)}_2 dt < \infty$, and
\begin{align*}
&\norm{\int_0^T \abs{f^2(t,Z^M_t,U^M_t) - f^2(t,Z^{M'}_t,U^{M'}_t)} dt}_2 \\
&\;\;\leq\rho \brak{\norm{Z^M}_{\bbH^2} + \norm{Z^{M'}}_{\bbH^2} 
	+ \norm{U^M}_{L^2(\tilde{N})} + \norm{U^{M'}}_{L^2(\tilde{N})}} 
\brak{\norm{Z^M-Z^{M'}}_{\bbH^2} + \norm{U^M-U^{M'}}_{L^2(\tilde{N})}},
\end{align*}
and $f^2(t,Z^M_t,U^M_t)$ is $\omega$-Lipschitz with constant
$C \brak{1 + \norm{Z^M_t}_2 + \norm{U^M_t}_{L^2(\bbP \otimes \mu)}}$.
\end{itemize}
\end{proposition}

\begin{proof}
Choose $\delta > 0$ so that 
$$
\sqrt{2 \delta} C \brak{1 + \sqrt{\lambda}} < \frac{1}{2} \quad \mbox{and} \quad k := T/\delta \in \bbN.
$$
Set $F^i_t(M) = \int_0^t f^i(s,Z^M_s,U^M_s) 1_{[T-\delta,T]}(s) ds$. It follows from the assumptions that
$F^i(M) \in \bbS^2_0$ for $i=1,2$ and all $M \in \bbM^2_0$. Moreover,
\beas
&& \norm{F^1_T(M) - \bbE_0 F^1_T(M) - (F^1_T(M') - \bbE_0 F^1_T(M'))}_2^2
\le \norm{F^1_T(M) - F^1_T(M')}_2^2\\
&& \le 2 \delta C^2 \int_{T-\delta}^T \brak{ \norm{Z^M_s - Z^{M'}_s}_2^2 
+ \norm{U^M_s -U^{M'}_s}^2_{L^2(\bbP \otimes \mu)}} ds 
\le 2 \delta C^2 \norm{M - M'}^2_{\bbS^2}.
\eeas
From condition (ii) one obtains that $M \in \mathbb{M}^2_0 \mapsto F^2_T(M) \in L^2({\cal F}_T)^d$ 
is continuous, and
\beas
&&
\abs{\int_{T-\delta}^T f^2(s,Z^M_s,U^M_s)(\omega) - f^2(s,Z^M_s,U^M_s)(\omega') ds}
\le \int_{T-\delta}^T \abs{f^2(s,Z^M_s,U^M_s)(\omega) - f^2(s,Z^M_s,U^M_s)(\omega')} ds\\
&& \le C \brak{\int_{T-\delta}^T \brak{1+\norm{Z^M_s}_2 + \norm{U^M_s}_{L^2(\bbP \otimes \mu)}} ds}
\norm{\omega - \omega'}\\
&& \le \brak{\delta C + \sqrt{\delta} C \sqrt{\int_{T-\delta}^T 2 \brak{\norm{Z^M_s}^2_2 
+ \norm{U^M_s}^2_{L^2(\bbP \otimes \mu)}} ds}} \norm{\omega - \omega'}\\
&&\le \brak{\delta C + \sqrt{2 \delta} C \norm{M}_{\bbS^2}} \norm{\omega - \omega'}.
\eeas
It follows that for all $M \in \mathbb{M}^2_0$, $F^2_T(M)$ is in $W^{1,2}(\Omega)^d$ with
$\norm{D F^2_T(M)}_2 \le \delta C + \sqrt{2 \delta} C \norm{M}_{\bbS^2}$.
So the conditions of Theorem \ref{thm:FM} hold with $\sqrt{2 \delta}C$ instead of $C$ and 
$\rho(x) = \delta C + \sqrt{2 \delta} Cx$. Therefore,
$$
Y_t = \xi + \int_t^T (f^1(s,Z^M_s,U^M_s) + f^2(s,Z^M_s,U^M_s)) 1_{[T-\delta,T]}(s) ds + M_T-M_t
$$
has a solution $(Y^{(k)},M^{(k)}) \in \bbS^2 \times \bbM^2_0$. From the same argument one obtains that, for $t \le T-\delta$, 
$$
Y_t = Y^{(k)}_{T-\delta} + \int_t^{T-\delta} (f^1(s,Z_s) + f^2(s,Z_s)) 1_{[T- 2 \delta, T-\delta]}(s) ds 
+ M_{T-\delta} - M_t$$
has a solution $(Y^{(k-1)},Z^{(k-1)}) \in \bbS^2 \times \bbM^2_0$.
Iterating this procedure, one obtains $(Y^{(j)},Z^{(j)})$, $j =1 , \dots, k$.
Now, define 
\begin{align*}
Y_t&:= Y^{(1)}_t,\quad M_t:= M^{(1)}_t \quad\text{for}\quad 0 \le t \le
\delta \quad\text{and}\quad\\
Y_t &:= Y^{(j)}_t,\quad
M_t - M_{(j-1) \delta}:= M^{(j)}_t -
M^{(j)}_{(j-1) \delta}
\end{align*}
for $(j-1) \delta < t \le j \delta$, $j =2, \dots, k$. Then $(Z^M_t, U^M_t) = (Z^{M^{(j)}}_t, U^{M^{(j)}}_t)$
for $(j-1) \delta < t \le j \delta$. So $(Y,M)$ is a solution of
\eqref{bsde:ZU} in $\bbS^2 \times \bbM^2_0$.
\end{proof}

As a consequence of Proposition \ref{prop:f1f2}, one obtains the following existence result for 
multidimensional mean-field BSDEs with drivers of quadratic growth and
square integrable terminal conditions. While there exist 
general existence and uniqueness results for one-dimensional BSDEs with 
drivers of quadratic growth (see e.g., Kobylanski, 2000, Briand and Hu, 2006, 2008,
or Delbaen et al., 2011), multidimensional quadratic BSDEs do not always admit solutions
(see Peng, 1999, or Frei and dos Reis, 2011). An existence and uniqueness result for 
multidimensional BSDEs with general drivers of quadratic growth was given by Tevzadze (2008). 
But it only holds for terminal conditions with small $L^{\infty}$-norm. Other results, such as the ones in Cheridito and Nam (2015), require the driver to have special structure.

\begin{corollary} \label{cor:f1f2}
Let $\xi \in L^2({\cal F}_T)^d$ and assume the driver is of the form 
$$
f(t,Z_t,U_t) = \tilde{\bbE} a(t,Z_t,\tilde{Z}_t,U_t,\tilde{U}_t) + B(t,\bbE b(t,Z_t,U_t))
$$
for mappings 
$a : [0,T] \times \Omega \times (\bbR^{d\times n})^2 \times (L^2(\mu))^2 \to \bbR^d$,
$b : [0,T] \times \Omega \times\bbR^{d\times n} \times L^2(\mu) \to\bbR^l$ and
$B : [0,T] \times \Omega \times\bbR^l \to\bbR^d$,
where $(\tilde{Z}_t,\tilde{U}_t)$ is a copy of $(Z_t,U_t)$ living on a separate 
probability space $(\tilde{\Omega}, \tilde{{\cal F}}, \tilde{\bbP})$, and
$\tilde{\bbE} a(t,Z_t,\tilde{Z}_t,U_t,\tilde{U}_t)$ means 
$\int_{\tilde{\Omega}} a(t,Z_t,\tilde{Z}_t,U_t,\tilde{U}_t) d\tilde{\bbP}$.

Then the BSDE \eqref{bsde:ZU} has a solution $(Y,M) \in \bbS^2 \times \bbM_0^2$
if there exists a constant $C \ge 0$ such that for all $z,\tilde{z}, z', \tilde{z}' \in \mathbb{R}^{d \times n}$, 
$u,\tilde{u},u' , \tilde{u}' \in L^2(\mu)$ and $x,x' \in \bbR^k$, 
$a(.,z,\tilde{z},u,\tilde{u})$, $b(.,z,u)$ and $B(.,x)$ are progressively measurable and 
the following hold:
\begin{itemize}
\item[{\rm (i)}]
$a(.,0,0,0,0) \in\bbH^2$ and 
$$
|a(t,z,\tilde{z},u,\tilde{u}) - a(t,z',\tilde{z}',u',\tilde{u}')|
\le C \brak{|z-z'| + |\tilde{z} - \tilde{z}'| + \norm{u-u'}_{L^2(\mu)} + \norm{\tilde{u} - \tilde{u}'}_{L^2(\mu)}}
$$
\item[{\rm (ii)}]
$|b(t,0,0)|, |B(t,0)| \le C$ and
at least one of the following two conditions is satisfied:
\begin{enumerate}
	\item[{(a)}] For any given $t\in[0,T],x,x'\in\bbR^l, z,z'\in\mathbb{R}^{d\times n},$ and $u,u'\in L^2(\mu)$, $B(t,x)$ is
	$\omega$-Lipschitz with constant $C(1+ \sqrt{|x|})$, and
	\begin{align*}
	&\abs{b(t,z,u)-b\bigl(t,z',u'\bigr)}\le C \bigl( 1+|z| + \abs{z'}+	 \norm{u}_{L^2(\mu)} + \norm{u'}_{L^2(\mu )} \bigr)\bigl( \bigl| z-z'\bigr| + \norm{u-u'}_{L^2(\mu)} \bigr),
	\\
	&|B(t,x)-B(t,x')| \le C | x-x'|.
	\end{align*}
	\item[{(b)}] For any given $t\in[0,T],x,x'\in\bbR^l, z,z'\in\mathbb{R}^{d\times n},$ and $u,u'\in L^2(\mu)$, $B(t,x)$ is
	$\omega$-Lipschitz with constant $C(1+ |  x|  )$, and
	\begin{align*}
	&| b(t,z,u)-b(t,z',u')| \leq C \bigl( \bigl| z-z'\bigr| + \norm{u-u'}_{L^2(\mu )} \bigr)\hspace{2.5in},
	\\
	&| B(t,x)-B(t,x')| \leq C(1+| x| +| x'| )| x-x'| .
	\end{align*}
\end{enumerate}
\end{itemize}
\end{corollary}
\begin{proof}
It is enough to show that 
$$
f^1(t,Z_t,U_t) := \tilde{\bbE} a(t,Z_t,\tilde{Z}_t,U_t, \tilde{U}_t) \quad \mbox{and} \quad
f^2(t,Z_t,U_t) := B(t,\bbE b(t,Z_t,U_t)) 
$$
satisfy the conditions of Proposition \ref{prop:f1f2}. As in the proof of Corollary 
\ref{cor:MKV}, one can deduce from Lemma 4.51 of Aliprantis and Border (2006)
that $f^i(t,Z_t,U_t)$ is progressively measurable and satisfies
$\int_0^T \norm{f^i(t,0,0)}_2 dt < \infty$ for $i=1,2$ and all 
$Z \in \bbH^2$ and $U \in L^2(\tilde{N})$.  

Now consider $Z,Z' \in \bbH^2$,
$U,U' \in L^2(\tilde{N})$, and let $(\tilde{Z}, \tilde{U},\tilde{Z}',\tilde{U}')$ be a copy of 
$(Z,U,Z',U')$ on $\tilde{\Omega}$. Then, for fixed $t \in [0,T]$,
\begin{align*}
&\bbE|\tilde{\bbE} a(t,Z_t,\tilde{Z}_t,U_t,\tilde{U}_t) - \tilde{\bbE} 
a(t,Z'_t,\tilde{Z}'_t,U'_t,\tilde{U}'_t)|^2 \le \bbE \tilde{\bbE} 
|a(t,Z_t,\tilde{Z}_t,U_t,\tilde{U}_t) - a(t, Z'_t,\tilde{Z}'_t,U'_t,\tilde{U}'_t)|^2\\
& \le 4 C^2 \brak{\bbE |Z_t - Z'_t|^2 + \tilde{\bbE}  |\tilde{Z}_t - \tilde{Z}'_t|^2
+ \bbE \norm{U_t - U'_t}^2_{L^2(\mu)} + \tilde{\bbE} \norm{\tilde{U}_t - \tilde{U}'_t}^2_{L^2(\mu)}}\\
&= 8 C^2 \brak{\norm{Z_t-Z'_t}^2_2 + \norm{U_t - U'_t}^2_{L^2(\bbP \otimes \mu)}}.
\end{align*}
On the other hand, if condition (ii.a) holds, then
\begin{align*}
& \norm{\int_0^T\abs{B(t,\bbE b(t,Z_t,U_t))- B(t,\bbE b(t,Z'_t,U'_t))}dt}_2 \\
&\leq C \int_0^T\abs{\bbE b(t,Z_t,U_t)-\bbE b(t,Z'_t,U'_t)}dt
\le C \bbE \int_0^T \abs{b(t,Z_t,U_t)- b(t,Z'_t,U'_t)} dt\\
\leq& C^2\bbE \int_0^T \brak{1+|Z_t|+|Z'_t| + \norm{U_t}_{L^2(\mu)} + \norm{U'_t}_{L^2(\mu)}}
\brak{|Z_t -Z'_t| + \norm{U_t - U'_t}_{L^2(\mu)}} dt\\
\le & C^2\sqrt{\bbE \int_0^T \brak{1+|Z_t|+|Z'_t| + \norm{U_t}_{L^2(\mu)}
+ \norm{U'_t}_{L^2(\mu)}}^2 dt} 
\sqrt{\bbE \int_0^T \brak{\abs{Z_t-Z'_t} + \norm{U_t - U'_t}_{L^2(\mu)}}^2 dt}\\
\leq & C^2 \sqrt{10} \, \sqrt{T+ \norm{Z}^2_{\bbH^2} + \norm{Z'}^2_{\bbH^2}
+ \norm{U}^2_{L^2(\tilde{N})} + \norm{U'}^2_{L^2(\tilde{N})}}
\sqrt{\norm{Z-Z'}^2_{\bbH^2} + \norm{U-U'}^2_{L^2(\tilde{N})}}\\
\le & C^2 \sqrt{10} \, \brak{\sqrt{T}+ \norm{Z}_{\bbH^2} + \norm{Z'}_{\bbH^2}
+ \norm{U}_{L^2(\tilde{N})} + \norm{U'}_{L^2(\tilde{N})}} 
\brak{\norm{Z-Z'}_{\bbH^2} + \norm{U-U'}_{L^2(\tilde{N})}}.
\end{align*}
Moreover, $B(t,\bbE b(t,Z_t,U_t))$ is $\omega$-Lipschitz with
constant $C(1 + \sqrt{|\bbE b(t,Z_t,U_t)|})$, and
\beas
&& |\bbE b(t,Z_t,U_t)| \le \bbE |b(t,Z_t,U_t)|\\
&\le& C \bbE \brak{1+|Z_t| + \norm{U_t}_{L^2(\mu)}} \brak{|Z_t| + \norm{U_t}_{L^2(\mu)}}\\
&\le& C\brak{1 + \norm{Z_t}_2 + \norm{U_t}_{L^2(\bbP \otimes \mu)}}
\brak{\norm{Z_t}_2 + \norm{U_t}_{L^2(\bbP \otimes \mu)}},
\eeas
from which one obtains that $B(t,\bbE b(t,Z_t,U_t))$ is $\omega$-Lipschitz with
constant $$
C(1 + \sqrt{C} (1+ \norm{Z_t}_2 + \norm{U_t}_{L^2(\bbP \otimes \mu)})).$$

Similarly, if condition (ii.b) holds, one has
\begin{align*}
&|B(t,\bbE b(t,Z_t,U_t))-B(t,\bbE b(t,Z'_t,U'_t)) |\\
&\leq C\brak{1 + \abs{\bbE b(t,Z_t,U_t)} + \abs{\bbE b(t,Z'_t,U'_t)}}
\abs{\bbE b(t,Z_t,U_t) - \bbE b(t,Z'_t,U'_t)}\\
&\leq C \brak{1 + \bbE \abs{b(t,Z_t,U_t)} +\bbE \abs{b(t,Z'_t,U'_t)}}
\bbE \abs{b(t,Z_t,U_t) - b(t,Z'_t,U'_t)}\\
&\leq C^2 \brak{1 + 2 \bbE |b(t,0,0)| + C \bbE \brak{|Z_t| + |Z'_t| 
+ \norm{U_t}_{L^2(\mu)} + \norm{U'_t}_{L^2(\mu)}}} \bbE \brak{|Z_t-Z'_t| + \norm{U_t-U'_t}_{L^2(\mu)}}. 
\end{align*}
Hence,
\begin{align*}
& \norm{\int_0^T\abs{B(t,\bbE b(t,Z_t,U_t))- B(t,\bbE b(t,Z'_t,U'_t))}dt}_2\\
\le& C^2 \sqrt{\int_0^T \brak{1+ 2C + C \bbE \brak{|Z_t| + |Z'_t| 
+ \norm{U_t}_{L^2(\mu)} + \norm{U'_t}_{L^2(\mu)}}}^2 dt}\\
& \qquad \qquad \times\sqrt{\int_0^T \brak{\bbE  |Z_t-Z'_t| + \bbE \norm{U_t-U'_t}_{L^2(\mu)}}^2 dt}\\
\le & C^3 \sqrt{\int_0^T 6 \brak{ C^{-2} + 4 + \norm{Z_t}_2^2 + \norm{Z'_t}_2^2 
+ \norm{U_t}^2_{L^2(\bbP \otimes \mu)} + \norm{U'_t}^2_{L^2(\bbP \otimes \mu)}} dt}\\
& \qquad \qquad \times\sqrt{\int_0^T 2 \brak{\norm{Z_t-Z'_t}^2_2 + \norm{U_t-U'_t}^2_{L^2(\bbP \otimes \mu)}} dt}\\
\le& C^3 \sqrt{12}
\brak{\sqrt{T( C^{-2} +4)} + \norm{Z}_{\bbH^2} + \norm{Z'}_{\bbH^2}
+ \norm{U}_{L^2(\tilde{N})}+ \norm{U'}_{L^2(\tilde{N})}}
\brak{\norm{Z-Z'}_{\bbH^2} + \norm{U-U'}_{L^2(\tilde{N})}}.
\end{align*}
Moreover, $B(t,\bbE b(s,Z_t,U_t))$ is $\omega$-Lipschitz with
constant $C(1 + |\bbE b(t,Z_t,U_t)|)$. So since 
\beas
|\bbE b(t,Z_t,U_t)| \le  \bbE |b(t,Z_t,U_t)|
\le C \brak{1 + \bbE \brak{|Z_t| + \norm{U_t}_{L^2(\mu)}}}\le C \brak{1 + \norm{Z_t}_2 + \norm{U_t}_{\bbP \otimes L^2(\mu)}},
\eeas
$B(t,\bbE b(t,Z_t,U_t))$ is $\omega$-Lipschitz with constant
$C \brak{1 + C  \brak{1 + \norm{Z_t}_2 + \norm{U_t}_{\bbP \otimes L^2(\mu)}}}$.
This shows that the conditions of Proposition \ref{prop:f1f2} hold, and the 
corollary follows.
\end{proof}

\begin{Example} \label{ex:multiquad} 
A simple example of a driver satisfying the conditions of Corollary \ref{cor:f1f2}
is given by 
$$
f(Z_t) = f^1(Z_t) + f^2(Z_t),
$$
for a Lipschitz function $f^1 : \mathbb{R}^{d \times n} \to \mathbb{R}^d$ and
a mapping $f^{2}:L^{2}(\scF_{T})^{d \times n} \to \bbR^{d}$ of the form
$$
f^{2}(Z_t):= \alpha + \bbE\brak{ Z_{t}|Z_t|}\beta
$$
with constant vectors $\alpha\in \mathbb{R}^{d\times 1}$ and $\beta\in\bbR^{n\times 1}$. In particular, if $W$ is an 
$n$-dimensional Brownian motion generating the filtration $\bbF$, the BSDE
$$
Y_t =\xi+\int_t^T f(Z_s) ds + \int_{t}^{T} Z_s dW_s
$$
has a solution $(Y,Z) \in \bbS^2 \times \bbH^2$ for every terminal condition $\xi \in L^2({\cal F}_T)^d$. 

Since $f^2$ has quadratic growth, the contraction mapping principle used by
Buckdahn et al. (2009) cannot be applied here. Also, if $d>1$ and $f^2$ were a function with quadratic 
growth of the realizations $Z_t(\omega)$, the existence of a global solution could not be guaranteed; 
see Frei and dos Reis (2011) for a counterexample.
\end{Example}


\begin{thebibliography}{10}
\bibitem{AB}
{Aliprantis, V. and Border, B. (2006). Infinite Dimensional
  Analysis. Springer, Berlin.}



\bibitem{BH}{Briand, P. and Hu, Y. (2006). BSDEs with quadratic growth and unbounded terminal value. 
Probability Theory and Related Fields 136(4), 604--618.}

\bibitem{BH2}{Briand, P. and Hu, Y. (2008). Quadratic BSDEs with convex 
generators and unbounded terminal conditions. Probability Theory and Related Fields 141(3-4), 543--567.}

\bibitem{Buckdahn:2009kq}
{Buckdahn, R., Li, J. and Peng, S. (2009). Mean-field backward stochastic differential equations and related
partial differential equations. Stochastic Process. Appl. 119(10), 3133--3154.}



   
\bibitem{Cheridito:2013tg}{Cheridito, P. and Nam, K. (2015). 
Multidimensional quadratic and subquadratic BSDEs with special
structure. Stochastics 87(5), 871--884.}

\bibitem{chita}{Chitashvili, R. (1983). Martingale ideology in the
theory of controlled stochastic processes. Lecture Notes in
Math. 1021, Springer-Verlag, Berlin, 73--92.}

\bibitem{DaPrato:2006vb}
{Da Prato, G. (2006). An Introduction to Infinite-Dimensional Analysis. Universitext. Springer-Verlag, Berlin.}



\bibitem{DHR}{Delbaen, F., Hu, Y. and Richou, A. (2011). On the uniqueness of solutions to quadratic 
BSDEs with convex generators and unbounded terminal conditions. Ann. Inst. H.
Poincar{\'e}, Probab. Stat. 47(2), 559--574.} 

\bibitem{DI}{Delong, L. and Imkeller, P. (2010a). Backward stochastic differential equations with 
time delayed generators -- results and counterexamples. Ann. Appl. Probab. 20(4), 1512--1536.}

\bibitem{DI2}{Delong, L. and Imkeller, P. (2010b). On Malliavin's
    differentiability of BSDEs with time delayed generators driven by
    Brownian motions and Poisson random measures. Stochastic Process. Appl. 120 (9), 1748--1775. }


\bibitem{FreidosReis2011} {Frei, C. and dos Reis, G. (2011).
A financial market with interacting investors: does an equilibrium exist?
Math. Financ. Econ. 4(3), 161--182. }



\bibitem{IkedaWatanabe}{
Ikeda, N. and Watanabe, S. (1989). Stochastic Differential Equations and Diffusion Processes. 
Second Edition. North-Holland Publishing Co., Amsterdam; Kodansha, Ltd., Tokyo.}

\bibitem{Jacod}
{Jacod, J. (1979). Calcul Stochastique et Probl\`emes de Martingales. 
Lectures Notes in Mathematics 714. Springer Verlag: Heidelberg.}


\bibitem{Kobylanski00}
{Kobylanski, M. (2000). Backward stochastic differential equations and partial differential
equations with quadratic growth.  Ann. Probab. 28(2), 558--602.}
  
\bibitem{Krasno00}{Krasnoselskii, M.A. (1964). 
Topological methods in the theory of nonlinear integral equations. Translated by A. H. Armstrong.
The Macmillan Co., New York.}

\bibitem{Liang:2011p33308}{
Liang, G., Lyons, T. and Qian, Z. (2011). Backward stochastic dynamics on a filtered probability space. 
Ann. Probab. 39(4), 1422--1448.}



\bibitem{maniatev}{Mania, M. and Tevzadze, R. (2003). A semimartingale
    backward equation and the variance-optimal martingale measure
    under general information flow. SIAM J. Control Optim. 42(5), 1703--1726.}

\bibitem{Pardoux:1990p20509}
{Pardoux, E. and Peng, S. (1990). Adapted solution of a backward stochastic differential equation. 
Systems Control Lett. 14(1), 55--61.}



\bibitem{Peng:1999ws}
{Peng, S. (1999).
Open problems on backward stochastic differential equations. Control of Distributed 
Parameter and Stochastic Systems. Kluwer Acad. Publ., Boston, MA.}



\bibitem{pengyang}{ Peng, S. and Yang, Z. (2009). Anticipated backward
  stochastic differential equations. Ann. Probab. 37(3), 877--902.}


\bibitem{Smart:1974tz}
{Smart, D.R. (1974). Fixed Point Theorems. 
Cambridge Tracts in Mathematics 66. Cambridge University Press.}


\bibitem{Tang:1994gu}{Tang, S. and Li, X. (1994). Necessary conditions for optimal control of stochastic 
systems with random jumps. SIAM J. Control Optimization, 32(5), 1447--1475.}

\bibitem{Tevzadze:2008p17046}
{Tevzadze, R. (2008). Solvability of backward stochastic differential equations with quadratic growth.
Stochastic Process. Appl. 118(3), 503--515.}
\end{thebibliography}
\end{document}